\documentclass[10pt,oneside,english]{amsart}
\usepackage[T1]{fontenc}
\usepackage[latin9]{inputenc}
\usepackage{geometry}
\usepackage{babel}
\usepackage{mathrsfs}
\usepackage{amstext}
\usepackage{amsthm}
\usepackage{amssymb}
\usepackage[all]{xy}
\usepackage[bookmarks=false,
 breaklinks=false,pdfborder={0 0 0},pdfborderstyle={},backref=false,colorlinks=false]
 {hyperref}

\makeatletter
\numberwithin{equation}{section}
\numberwithin{figure}{section}
\theoremstyle{plain}
\newtheorem{thm}{\protect\theoremname}
\theoremstyle{plain}
\newtheorem{prop}[thm]{\protect\propositionname}
\theoremstyle{plain}
\newtheorem{lem}[thm]{\protect\lemmaname}
\theoremstyle{remark}
\newtheorem{rem}[thm]{\protect\remarkname}
\theoremstyle{plain}
\newtheorem{cor}[thm]{\protect\corollaryname}
\theoremstyle{definition}
\newtheorem{example}[thm]{\protect\examplename}

\usepackage{xypic}
\usepackage{amstext}
\usepackage{amsthm}
\usepackage{amssymb}
\usepackage{amscd}
\usepackage{graphicx}
\usepackage{mathrsfs, setspace, fancyhdr, times, bm}

\makeatother

\providecommand{\corollaryname}{Corollary}
\providecommand{\examplename}{Example}
\providecommand{\lemmaname}{Lemma}
\providecommand{\propositionname}{Proposition}
\providecommand{\remarkname}{Remark}
\providecommand{\theoremname}{Theorem}

\begin{document}
\title[$X_{22}$ with infinite automorphism groups]{Another view on smooth prime Fano threefolds of degree $22$ with
infinite automorphism groups}
\author{Adrien Dubouloz}
\address{Laboratoire de Mathématique et Applications, UMR 7348 CNRS, Université
de Poitiers, 86000 Poitiers, France. \vspace{-1em}}
\address{Université Bourgogne Europe, CNRS, IMB UMR 5584, 21000 Dijon, France}
\email{adrien.dubouloz@math.cnrs.fr}
\author{Kento Fujita}
\address{Department of Mathematics, Graduate School of Science, Osaka University,
Toyonaka, Osaka 560-0043, Japan}
\email{fujita@math.sci.osaka-u.ac.jp}
\author{Takashi Kishimoto}
\address{Department of Mathematics, Faculty of Science, Saitama University,
Saitama 338-8570, Japan}
\email{tkishimo@rimath.saitama-u.ac.jp}
\dedicatory{Dedicated to Yuri Prokhorov}
\begin{abstract}
We give a self-contained alternative proof of the classification of
smooth prime Fano threefolds of degree $22$ with infinite automorphism
groups established by Kuznetsov, Prokhorov and Shramov. 
\end{abstract}

\maketitle
\vspace{-0.8cm}

\section*{Introduction}

Smooth prime Fano threefolds $X$ of degree $(-K_{X})^{3}=22$ over
an algebraically closed field $k$ of characteristic zero form a family
${\mathscr{M}}_{22}$ of dimension 6 whose members are rational \cite{Mu04}.
Our main interest lies in the subfamily ${\mathscr{M}}_{22}^{\circ}$
consisting of threefolds with infinite automorphism groups, whose
classification was established by Kuznetsov, Prokhorov and Shramov
in \cite{KP18,KPS18} in a form which can be summarized as follows
(see Theorem \ref{prop:Second-Isomorphism-Classification}):
\begin{thm}
\label{thm:KPS} A smooth prime Fano threefold of degree $22$ belonging
to ${\mathscr{M}}_{22}^{\circ}$ is isomorphic to one of the following
threefolds:

$\bullet$ A threefold $X_{22}^{{\rm MU}}\in{\mathscr{M}}_{22}^{\circ}$
with ${\rm Aut}^{0}(X_{22}^{{\rm MU}})={\rm Aut}(X_{22}^{{\rm MU}})\cong{\rm PGL}_{2}$
(the Mukai-Umemura threefold),

$\bullet$ A threefold $X_{22}^{a}\in{\mathscr{M}}_{22}^{\circ}$
with ${\rm Aut}^{0}(X_{22}^{a})\cong\mathbb{G}_{a}$. Moreover, ${\rm Aut}(X_{22}^{a})$
is isomorphic to a subgroup $\mathbb{G}_{a}\rtimes\mu_{4}$ of a Borel
subgroup $\mathbb{G}_{a}\rtimes\mathbb{G}_{m}$ of $\mathrm{PGL}_{2}$. 

$\bullet$ A member of a one-parameter family of pairwise non isomorphic
threefold $X_{22}^{m}(v)\in{\mathscr{M}}_{22}^{\circ}$ with ${\rm Aut}^{0}(X_{22}^{m}(v))\cong\mathbb{G}_{m}$
parametrized by $v\in{\mathbb{P}}^{1}\backslash\{0,1,\infty,-4\}$.
Moreover, for every $v\in{\mathbb{P}}^{1}\backslash\{0,1,\infty,-4\}$,
${\rm Aut}(X_{22}^{m}(v))$ is isomorphic to the normalizer $\mathbb{G}_{m}\rtimes\mu_{2}$
of a maximal torus $\mathbb{G}_{m}$ of $\mathrm{PGL}_{2}$ .
\end{thm}

The possibilities for $X$ were initially found by Prokhorov in \cite{Pr90}
and the articles \cite{KPS18} and \cite{KP18} completed later on
the structures of the automorphism groups ${\rm Aut}(X_{22}^{a})$
and ${\rm Aut}(X_{22}^{m}(v))$. Our purpose is to give a self-contained
alternative proof of Theorem \ref{thm:KPS} and a complementary view
on certain of the constructions and results in these articles. 

\medskip

Our approach to the classification of isomorphism types of prime Fano
threefolds of degree $22$ with infinite automorphism follows the
same initial path as in \cite{KP18,KPS18,Pr90}. It builds essentially
on Iskovskikh's \emph{double projection from lines} \cite{Isk89}
which are birational maps relating the geometry of $X$ to that of
the smooth quintic del Pezzo threefold $V_{5}$ in $\mathbb{P}^{6}$.
The principle stems from the observation that such a double projection
$\psi_{Z}:X\dashrightarrow V_{5}$ from a line $Z$ in $X$ induces
a correspondence between subgroups of the stabilizer $\mathrm{Aut}(X,Z)$
of $Z$ which stabilize the base locus of $\psi_{Z}$ and subgroups
of the stabilizer $\mathrm{Aut}(V_{5},\Gamma)$ of a rational normal
quintic curve $\Gamma\subset V_{5}$ contained in the base locus of
$\psi_{Z}^{-1}$ which stabilize the base locus of $\psi_{Z}^{-1}$
. This leads to determine equivalence classes of rational normal quintic
curves in $V_{5}$ with infinite stabilizers under the action of $\mathrm{Aut}(V_{5})$.
In \cite{KPS18}, this is done by exhibiting representatives of these
classes in the form of orbit closures of certain polynomials in Mukai-Umemura's
description \cite{MuUm83} of $V_{5}$ as an orbit closure of the
action of $\mathrm{PGL}_{2}$ on the projective space of homogeneous
polynomials of degree $6$ in two variables. We propose a complementary
approach building on the study of the equivariant geometry of inverse
images of rational normal quintic curves in $V_{5}$ in the universal
family on the Hilbert scheme of lines on $V_{5}$ constructed by Furushima-Nakayama
\cite{FuNa89}. 

To determine the exact automorphism group for each isomorphism type,
the main difficulty essentially reduces to establish the existence
of a special involution of $X$ generating the displayed subgroup
$\mu_{2}$ in the case $X=X_{22}^{m}(v)$. In \cite{KPS18}, the existence
is taken for granted by an external result \cite[Proposition 5.1]{DKK}
which depends itself on Mukai's theory of varieties of sums of powers.
Later, \cite{KP18} provided a different argument for the existence
which depends in particular on the fact established in \cite{KPS18}
that the Hilbert schemes of conics in $X$ is isomorphic to $\mathbb{P}^{2}$.
In contrast, we produce such a special involution in a geometric way,
as a conjugate by an appropriate $\mathbb{G}_{m}$-equivariant birational
map of a known involution on certain quadric threefolds with $\mathbb{G}_{m}$-actions
normalizing the $\mathbb{G}_{m}$-action, see the proof of Theorem
\ref{prop:Second-Isomorphism-Classification} and Lemma \ref{prop:The-crucial-involution}
as well as Remark \ref{rem:Link-remark} for these constructions. 

\medskip

We henceforth work over a fixed algebraically closed base field $k$
of characteristic zero. The scheme of the article is the following.
In Section \ref{sec:V5-Stuff}, we first review basic properties of
the construction of the Hilbert scheme of lines on $V_{5}$ and then
proceed to the classification of rational normal quintic curves $\Gamma\subset V_{5}$
with infinite stabilizers up to the action of the automorphism group
of $V_{5}$. In Section \ref{sec:X_22-stuff}, we begin with a review
of classical properties of the double projection from a line in a
smooth prime Fano threefold $X$ of degree $22$ and of the correspondence
it induces between certain families of lines in $X$ and $V_{5}$
and then re-derive Theorem \ref{thm:KPS} from these results.\\

\textit{Acknowledgements.} We would like to express our sincere gratitude
to Sasha Kuznetsov for his thorough readings of successive drafts
of this article and his numerous constructive comments and suggestions,
which contributed significantly to the clarity and mathematical accuracy
of the results. 

The present research was initiated during visits of the first and
the second authors at Saitama University and continued during a one-month
visit of the second and third authors at the University of Poitiers
during fall 2024 funded by the University of Poitiers . The authors
are grateful to these institutions for their generous support and
the excellent working conditions offered. The second author was supported
by JSPS KAKENHI Grant Number 22K03269, Royal Society International
Collaboration Award ICA\textbackslash 1\textbackslash 23109 and
Asian Young Scientist Fellowship. The third author was partially funded
by JSPS KAKENHI Grant Number 23K03047. 

\section{\protect\label{sec:V5-Stuff}Rational quintic curves with infinite
stabilizers on the smooth quintic del Pezzo threefold}

A \emph{quintic del Pezzo threefold} is a smooth projective threefold
$V_{5}$ whose Picard group is isomorphic to $\mathbb{Z}$, generated
by a very ample invertible sheaf $\mathcal{L}$ such that $\omega_{V_{5}}^{\vee}\cong\mathcal{L}^{\otimes2}$
and $\deg c_{1}(\mathcal{L})^{3}=5$, where $\omega_{V_{5}}=\mathcal{O}_{V_{5}}(K_{V_{5}})$
denotes the canonical sheaf of $V_{5}$. The global sections of $\mathcal{L}$
determine a closed embedding $\varphi_{|\mathcal{L}|}:V_{5}\hookrightarrow\mathbb{P}(H^{0}(V_{5},\mathcal{L}))\cong\mathbb{P}^{6}$
of $V_{5}$ as a codimension $3$ subvariety of $\mathbb{P}^{6}$
isomorphic to a section of the Grassmannian $\mathbb{G}(5,2)\subset\mathbb{P}^{9}$
by a linear subspace of codimension $3$. Conversely, every smooth
codimension $3$ linear section of $\mathbb{G}(5,2)$ is a quintic
del Pezzo threefold, and all of these are isomorphic, see e.g. \cite{Fuj81}.
We henceforth denote by $V_{5}$ any smooth quintic del Pezzo threefold
and consider it as a subvariety of $\mathbb{P}^{6}$ via its half-anti-canonical
embedding above. A subvariety $Z$ of $V_{5}$ of degree $d$ is then
by definition a subvariety $Z\subset\mathbb{P}^{6}$ of degree $d$
which is contained in $V_{5}$. 

\subsection{\protect\label{subsec:Automorphisms-and-induced-action-on-lines-V5}
Basic facts on lines and rational normal quintic curves }

\subsubsection{\protect\label{subsec:Lines-V5}Equivariant structure of the Hilbert
scheme of lines }

A line on $V_{5}$ is a usual line $\ell\cong\mathbb{P}^{1}$ in $\mathbb{P}^{6}$
which is contained in $V_{5}$, equivalently, an integral curve $\ell\subset V_{5}$
of anti-canonical degree $-K_{V_{5}}\cdot\ell=\deg(c_{1}(\omega_{V_{5}}^{\vee}|_{\ell}))=2$.
By \cite[Proposition 5.1]{Isk77}, the conormal sheaf of $\ell$ in
$V_{5}$ is either trivial, in which case $\ell$ is said to be a
line of general type, or isomorphic to $\mathcal{O}_{\ell}(-1)\oplus\mathcal{O}_{\ell}(1)$,
in which case $\ell$ is said to be a line of special type. We now
recall from \cite{FuNa89} some properties of the Hilbert scheme of
lines on $V_{5}$ and of its universal family, see also \cite{Sa17}
for another view on this matter. 

\medskip

Let $\mathfrak{C}\cong\mathbb{P}^{1}$ be a smooth proper $k$-rational
curve, let $\omega_{\mathfrak{C}}^{\vee}\cong\mathcal{O}_{\mathfrak{C}}(2)$
be endowed with its canonical $\mathrm{Aut}(\mathfrak{C})$-linearization
and let $V=H^{0}(\mathfrak{C},\omega_{\mathfrak{C}}^{\vee})$ be endowed
with the corresponding structure of $\mathrm{Aut}(\mathfrak{C})$-module.
The anti-canonical embedding $\mathfrak{C}\to\mathfrak{H}:=\mathbb{P}(V)\cong\mathbb{P}^{2}$
is then $\mathrm{Aut}(\mathfrak{C})$-equivariant, and, identifying
$\mathfrak{C}$ with its image, which is a smooth conic, the restriction
homomorphism $\mathrm{Aut}(\mathfrak{H},\mathfrak{C})\to\mathrm{Aut}(\mathfrak{C})$
is an isomorphism. Let $\mathfrak{R}:=\mathfrak{C}\times\mathfrak{C}$
be endowed with the diagonal action of $\mathrm{Aut}(\mathfrak{C})$
so that the two projections $\mathrm{p}_{i}:\mathfrak{R}\to\mathfrak{C}$,
$i=1,2$, are $\mathrm{Aut}(\mathfrak{C})$-equivariant and let $\Delta:\mathfrak{C}\to\mathfrak{R}$
be the diagonal closed immersion. The irreducible sub-$\mathrm{Aut}(\mathfrak{C})$-module
$V$ of the $\mathrm{Aut}(\mathfrak{C})$-module $H^{0}(\mathfrak{R},\mathcal{O}_{\mathfrak{R}}(\Delta(\mathfrak{C})))\cong V\oplus k$
determines an $\mathrm{Aut}(\mathfrak{C})$-equivariant double cover
$\upsilon:\mathfrak{R}\to\mathfrak{H}$ ramified along $\Delta(\mathfrak{C})$
and branched along $\mathfrak{C}$. The sheaf $\mathcal{E}:=\upsilon_{*}(\mathrm{p}_{1}^{*}(\omega_{\mathfrak{C}}^{\vee})^{\otimes2})$
on $\mathfrak{H}$ is then an $\mathrm{Aut}(\mathfrak{C})$-linearized
locally free sheaf of rank $2$ (cf. \cite[Theorem 5]{Sch61}), with
associated $\mathrm{Aut}(\mathfrak{C})$-equivariant projective bundle
$\pi:\mathfrak{U}:=\mathbb{P}(\mathcal{E})\to\mathfrak{H}$. Taking
cohomology of the exact sequence of $\mathrm{Aut}(\mathfrak{C})$-linearized
sheaves 
\[
0\to\mathcal{O}_{\mathfrak{R}}\to\mathcal{O}_{\mathfrak{R}}(\Delta(\mathfrak{C}))\to\Delta_{*}\omega_{\mathfrak{C}}^{\vee}\to0
\]
 tensored by $\mathrm{p}_{1}^{*}(\omega_{\mathfrak{C}}^{\vee})^{\otimes2}$
gives an isomorphism of $\mathrm{Aut}(\mathfrak{C})$-modules 
\[
H^{0}(\mathfrak{R},\mathrm{p}_{1}^{*}(\omega_{\mathfrak{C}}^{\vee})^{\otimes2}\otimes\mathcal{O}_{\mathfrak{R}}(\Delta(\mathfrak{C})))\cong H^{0}(\mathfrak{C},(\omega_{\mathfrak{C}}^{\vee})^{\otimes3})\oplus H^{0}(\mathfrak{C},(\omega_{\mathfrak{C}}^{\vee})^{\otimes2}).
\]
Under the isomorphism of $\mathrm{Aut}(\mathfrak{C})$-modules 
\[
H^{0}(\mathfrak{H},\mathcal{E}\otimes\mathcal{O}_{\mathcal{\mathfrak{H}}}(1))\cong H^{0}(\mathfrak{R},\mathrm{p}_{1}^{*}(\omega_{\mathfrak{C}}^{\vee})^{\otimes2}\otimes\upsilon^{*}\mathcal{O}_{\mathcal{\mathfrak{H}}}(1))\cong H^{0}(\mathfrak{R},\mathrm{p}_{1}^{*}(\omega_{\mathfrak{C}}^{\vee})^{\otimes2}\otimes\mathcal{O}_{\mathfrak{R}}(\Delta(\mathfrak{C}))),
\]
the irreducible sub-$\mathrm{Aut}(\mathfrak{C})$-module $W:=H^{0}(\mathfrak{C},(\omega_{\mathfrak{C}}^{\vee})^{\otimes3})$
determines an $\mathrm{Aut}(\mathfrak{C})$-equivariant rational map
\[
\Psi:\mathfrak{U}\dashrightarrow\mathbb{P}(W)\cong\mathbb{P}^{6}.
\]
By \cite[Lemma 2.2 and Lemma 2.3 (1)]{FuNa89}, $\Psi$ is a morphism
with image isomorphic to $V_{5}$ and by \cite[Theorem I]{FuNa89},
the induced morphism $\psi:\mathfrak{U}\to V_{5}$ identifies $\mathfrak{H}$
with the Hilbert scheme of lines of $V_{5}$ and the projective bundle
$\text{\ensuremath{\pi}}:\mathfrak{U}\to\mathfrak{H}$ with its universal
family. The following proposition collects additional basic properties
of the construction. 
\begin{prop}
\label{prop:Hilb-Evaluation-Morphism}With the notation above, the
following hold:

a) The canonical surjection $\upsilon^{*}\mathcal{E}\to\mathrm{p}_{1}^{*}(\omega_{\mathfrak{C}}^{\vee})^{\otimes2}$
induces an $\mathrm{Aut}(\mathfrak{C})$-equivariant closed immersion
$j:\mathfrak{R}\hookrightarrow\mathfrak{U}$ of $\mathfrak{H}$-schemes
for which the composition $\Psi_{\mathfrak{R}}:=\Psi\circ j:\mathfrak{R}\to\mathbb{P}(W)$
is an injective morphism with image equal to an anti-canonical divisor
$\mathcal{S}$ of $V_{5}$ singular along the the rational normal
sextic curve $\mathcal{C}=\Psi_{\mathfrak{R}}(\Delta(\mathfrak{C}))$.
Moreover, the induced morphism $\psi_{\mathfrak{R}}:\mathfrak{R}\to\mathcal{S}$
is the normalization of $\mathcal{S}$. 

b) The image by $\Psi_{\mathfrak{R}}$ of a fiber $f_{i}$ of $\mathrm{p}_{i}:\mathfrak{R}\to\mathfrak{C}$
is a line of special type of $V_{5}$ for $i=1$ and a rational normal
quintic curve for $i=2$.

c) The restriction $\mathcal{E}|_{\mathfrak{C}}$ is isomorphic to
$\mathcal{O}_{\mathfrak{C}}(3)\oplus\mathcal{O}_{\mathfrak{C}}(3)$
and the restriction $\Psi:\mathfrak{U}|_{\mathfrak{C}}\to\mathbb{P}(W)$
is injective, with image $\mathcal{S}$. A line $\ell\subset V_{5}$
is of special type if and only if it equals the image by $\Psi$ of
a fiber of $\pi:\mathfrak{U}|_{\mathfrak{C}}\to\mathfrak{C}$.

d) The morphism $\psi:\mathfrak{U}\to V_{5}$ is finite of degree
$3$, its restriction over $V_{5}\setminus\mathcal{S}$ is étale and
$\psi^{*}\mathcal{S}=2\mathfrak{U}|_{\mathfrak{C}}+j(\mathfrak{R})$.
Moreover, $\mathfrak{U}|_{\mathfrak{C}}\cap j(\mathfrak{R})=2j(\Delta(\mathfrak{C}))$
scheme-theoretically and $j(\Delta(\mathfrak{C}))\subset\mathcal{\mathfrak{U}}|_{\mathfrak{C}}$
is an $\mathrm{Aut}(\mathfrak{C})$-stable section of $\pi:\mathfrak{U}|_{\mathfrak{C}}\to\mathfrak{C}$. 
\end{prop}

\begin{proof}
In assertion a), the existence of $j$ and the fact that the image
of $\Psi_{\mathfrak{R}}$ is an anti-canonical divisor of $V_{5}$
follow from \cite[Lemma 2.1 (3)-(4)]{FuNa89}, which identify in particular
the image of $j$ as the zero locus of an $\mathrm{Aut}(\mathfrak{C})$-invariant
global section of $\mathcal{O}_{\mathbb{P}(\mathcal{E})}(2)\otimes\pi^{*}\mathcal{O}_{\mathcal{\mathfrak{H}}}(-2)$
endowed with the product $\mathrm{Aut}(\mathfrak{C})$-linearization.
The other properties follow from \cite[Remark 2.2 (2)]{FuNa89}. Assertion
b) is for instance directly verified from the following explicit basis
of the sub-$\mathrm{Aut}(\mathfrak{C})$-module $W$ of $H^{0}(\mathfrak{R},\mathrm{p}_{1}^{*}(\omega_{\mathfrak{C}}^{\vee})^{\otimes2}\otimes\mathcal{O}_{\mathfrak{R}}(\Delta(\mathfrak{C})))\cong H^{0}(\mathfrak{R},\mathcal{O}_{\mathfrak{R}}(5f_{1}+f_{2}))$
given in \cite[$\S$ 2, p. 116]{FuNa89} 
\begin{align}
e_{0}=x_{1}^{5}x_{2},\,e_{1}=x_{1}^{4}y_{1}x_{2}+\tfrac{1}{5}x_{1}^{5}y_{2} & ,\,e_{2}=x_{1}^{3}y_{1}^{2}x_{2}+\tfrac{1}{2}x_{1}^{4}y_{1}y_{2},\,e_{3}=x_{1}^{2}y_{1}^{3}x_{2}+x_{1}^{3}y_{1}^{2}y_{2},\label{eq:basis-FN}\\
e_{4}=\tfrac{1}{2}x_{1}y_{1}^{4}x_{2}+x_{1}^{2}y_{1}^{3}y_{2},\, & e_{5}=\tfrac{1}{5}y_{1}^{5}x_{2}+x_{1}y_{1}^{4}y_{2},\,e_{6}=y_{1}^{5}y_{2}\nonumber 
\end{align}
where $[x_{1}:y_{1}]$ and $[x_{2}:y_{2}]$ are homogeneous coordinates
on the first and second factor of $\mathfrak{R}=\mathfrak{C}\times\mathfrak{C}\cong\mathbb{P}^{1}\times\mathbb{P}^{1}$
respectively. These sections $e_{i}$, $i=0,\ldots,6$, form a basis
of the weight space decomposition of $W$ with respect to the diagonal
 action $([x_{1}:y_{1}],[x_{2}:y_{2}])\mapsto([\lambda x_{1}:y_{1}],[\lambda x_{2}:y_{2}])$
of the maximal torus $\mathbb{G}_{m}=\mathrm{Spec}(k[\lambda^{\pm1}])$
of $\mathrm{Aut}(\mathfrak{C})\cong\mathrm{PGL}_{2}$ with respective
weights $w_{i}=6-i$. Assertion c) is a combination of \cite[Lemma 2.1 (2) and Lemma 2.3 (3)-(4)]{FuNa89}.
The first part of assertion d) follows from \cite[Lemma 2.3 (2)-(3)]{FuNa89}
and second part from the previously explained identification of $j(\mathfrak{R})$
to the zero locus of an $\mathrm{Aut}(\mathfrak{C})$-invariant global
section of $\mathcal{O}_{\mathbb{P}(\mathcal{E})}(2)\otimes\pi^{*}\mathcal{O}_{\mathcal{\mathfrak{H}}}(-2)$
.
\end{proof}
In geometric terms, $V_{5}\setminus\mathcal{S}$ is the locus of points
of $V_{5}$ through which there pass exactly $3$ lines, all of general
type, $\mathcal{C}$ is the locus of points of $V_{5}$ through which
there passes a unique line of $V_{5}$, of special type, $\mathcal{S}$
is the surface swept out by lines of special type on $V_{5}$ and
$\mathcal{S}\setminus\mathcal{C}$ is the locus of points through
which pass a line of special type and a line of general type. Moreover,
lines of special types are tangent lines to $\mathcal{C}$ whereas
every line $\ell$ of general type intersect $\mathcal{S}$ transversally
in two distinct points, contained in $\mathcal{S}\setminus\mathcal{C}$,
the later being the images by $\psi:\mathfrak{U}\to V_{5}$ of the
two intersection points with $j(\mathfrak{R})$ of the corresponding
fiber of $\pi:\mathfrak{U}\to\mathfrak{H}$. 

The surface $\mathcal{S}$ is also swept out by the images $\Psi_{\mathfrak{R}}(f_{2})$
of the fibers of $\mathrm{p}_{2}:\mathfrak{R}\to\mathfrak{C}$, which
are rational normal quintic curves in $\mathbb{P}(W)$. These curves
are disjoint on $V_{5}$, meeting every line of special type in pairwise
distinct points. We henceforth call them special rational normal quintic
curves of $V_{5}$. 

\medskip

Proposition \ref{prop:Hilb-Evaluation-Morphism} a) implies in particular
that the closed subschemes $\mathcal{C}$ and $\mathcal{S}$ of $V_{5}$
are invariant under the action of $\mathrm{Aut}(V_{5})$. Since $\mathcal{C}=\Psi_{\mathfrak{R}}(\Delta(\mathfrak{C}))\subset V_{5}\subset\mathbb{P}(W)$
is a rational normal sextic curve, the restriction homomorphism $\mathrm{Aut}(\mathbb{P}(W),\mathcal{C})\to\mathrm{Aut}(\mathcal{C})\cong\mathrm{Aut}(\mathfrak{C})$
is an isomorphism and it follows from \cite{MuUm83} that 
\begin{align*}
\mathrm{Aut}(\mathbb{P}(W),V_{5})= & \mathrm{Aut}(\mathbb{P}(W),(V_{5},\mathcal{S},\mathcal{C}))=\mathrm{Aut}(\mathbb{P}(W),(\mathcal{S},\mathcal{C}))=\mathrm{Aut}(\mathbb{P}(W),\mathcal{C}),
\end{align*}
and that the restriction homomorphism $\mathrm{Aut}(\mathbb{P}(W),\mathcal{C})\to\mathrm{Aut}(V_{5})$
is an isomorphism. It can be verified in turn that the action of $\mathrm{Aut}(V_{5})$
on $V_{5}$ has exactly three orbits: $\mathcal{C}$, $\mathcal{S}\setminus\mathcal{C}$
and $V_{5}\setminus\mathcal{S}$. 

\medskip

The construction provides the following commutative diagram of $\mathrm{Aut}(\mathfrak{C})$-equivariant
morphisms 

\[\xymatrix{ & \mathcal{C} \ar[r] & \mathcal{S} \ar[r] & V_5 \\ \mathfrak{C} \ar[ur]^{\cong} \ar[r]^{\Delta} \ar[d]_{\cong} & \mathfrak{R} \ar[ur]^{\psi_{\mathfrak{R}}} \ar[d]_{\upsilon} \ar[r]^{j} & \mathfrak{U} \ar[ur]^{\psi} \ar[dl]_{\pi} \\ \mathfrak{C} \ar[r] & \mathfrak{H}}\]and
two canonical identifications 
\begin{equation}
c_{\mathcal{C}}:\mathrm{Aut}(V_{5})\to\mathrm{Aut}(\mathcal{C})\quad\textrm{and}\quad c_{\mathcal{\mathfrak{C}}}:\mathrm{Aut}(V_{5})\to\mathrm{Aut}(\mathfrak{C}),\label{eq:canonical-ID-Aut(V5)}
\end{equation}
the first one given by the restriction homomorphism $\mathrm{Aut}(V_{5},\mathcal{C})\to\mathrm{Aut}(\mathcal{C})$
and the second one by the restriction homomorphism $\mathrm{Aut}(\mathfrak{H},\mathfrak{C})\to\mathrm{Aut}(\mathfrak{C})$
when viewing $\mathrm{Aut}(V_{5})$ as the subgroup $\mathrm{Aut}(\mathfrak{H},\mathfrak{C})$
of $\mathrm{Aut}(\mathcal{\mathfrak{H}})$. 

\medskip

Recall that a Borel subgroup of an affine algebraic group is a maximal
connected solvable subgroup. We will make frequent use of the fact
that a nontrivial proper connected subgroup of $\mathrm{Aut}(V_{5})\cong\mathrm{Aut}(\mathfrak{C})\cong\mathrm{PGL}_{2}$
is either a Borel subgroup $B\cong\mathbb{G}_{a}\rtimes\mathbb{G}_{m}$
or its unipotent radical isomorphic to $\mathbb{G}_{a}$, or a maximal
torus $T$ isomorphic to $\mathbb{G}_{m}$ and that all these groups
are in particular solvable. We record the following well-known consequence
of the above properties: 
\begin{lem}
\label{lem:Connected-Group-invariant-lines-V5}For a non-trivial proper
connected subgroup $G$ of $\mathrm{Aut}(V_{5})$, the following hold:

1) If $G$ is a torus then $V_{5}$ contains three $G$-stable lines:
two of special type and one of general type.

2) Otherwise, $V_{5}$ contains a unique $G$-stable line, which is
of special type. 
\end{lem}

\begin{proof}
Since $c_{\mathfrak{C}}(G)\cong G$ is a proper connected subgroup
of $\mathrm{Aut}(\mathfrak{C})$, $\mathfrak{C}$ contains two $G$-fixed
points if $G$ is a torus and a unique one otherwise. On the other
hand, every $G$-fixed point $q$ in $\mathfrak{H}\setminus\mathfrak{C}$
determines two distinct $G$-fixed points on $\mathfrak{C}$ (the
intersections of $\mathfrak{C}$ with its tangent lines passing through
$q$). The assertion then follows from the fact that a pair of distinct
$G$-fixed points in $\mathfrak{H}\setminus\mathfrak{C}$ determines
in a similar way at least three distinct $G$-fixed points on $\mathfrak{C}$. 
\end{proof}

\subsubsection{\protect\label{subsec:Normalization-Special-Hyperplane-Section}The
hyperplane section singular along a line of special type and its normalization}

By Proposition \ref{prop:Hilb-Evaluation-Morphism} b), every point
$p$ of the rational normal sextic curve $\mathcal{C}\subset V_{5}$
determines a unique special rational normal quintic $\Gamma_{p}\subset V_{5}$
and a unique line $\ell_{p}$ of special type passing through $p$,
both contained in the surface $\mathcal{S}$. 
\begin{lem}
\label{lem:non-normal-section-charac}There exists a unique hyperplane
section $F_{p}$ of $V_{5}$ which is singular along $\ell_{p}$.
The surface $F_{p}$ is swept out by the lines in $V_{5}$ intersecting
$\ell_{p}$. The surface $F_{p}$ contains $\Gamma_{p}$, $F_{p}\cap\mathcal{S}=5\ell_{p}\cup\Gamma_{p}$,
$F_{p}\cap\mathcal{C}=6p$ and the stabilizer $\mathrm{Aut}(V_{5},F_{p})$
of $F_{p}$ in $\mathrm{Aut}(V_{5})$ is the inverse image by $c_{\mathcal{C}}:\mathrm{Aut}(V_{5})\to\mathrm{Aut}(\mathcal{C})$
of the stabilizer $G_{p}\cong\mathbb{G}_{a}\rtimes\mathbb{G}_{m}$
of the point $p$. Moreover, the restriction homomorphism $\mathrm{Aut}(V_{5},F_{p})\to\mathrm{Aut}(F_{p})=\mathrm{Aut}(F_{p},\ell_{p},p)$
is injective. 
\end{lem}

\begin{proof}
The existence, the uniqueness and the fact that $F_{p}$ is swept
out by lines in $V_{5}$ intersecting $\ell_{p}$ follow from \cite[Proposition (9.11)]{Fuj81}
which identifies $F_{p}$ to the exceptional locus of the projection
$V_{5}\dashrightarrow Q\subset\mathbb{P}^{4}$ from $\ell_{p}$, where
$Q\subset\mathbb{P}^{4}$ is a smooth quadric threefold. The uniqueness
of $F_{p}$ implies that it is stable under the action of the inverse
image $c_{\mathcal{C}}^{-1}(G_{p})\subset\mathrm{Aut}(V_{5})$ of
the stabilizer $G_{p}\subset\mathrm{Aut}(\mathcal{C})$ of the point
$p$. Since $G_{p}$ acts transitively on $\mathcal{C}\setminus\{p\}$,
we have $F_{p}\cap\mathcal{C}=6p$ which implies in turn that $\mathrm{Aut}(V_{5},F_{p})=c_{\mathcal{C}}^{-1}(G_{p})$.
Letting $H$ be the unique hyperplane of $\mathbb{P}(W)$ such that
$F_{p}=V_{5}\cap H$, it follows from Proposition \ref{prop:Hilb-Evaluation-Morphism}
a) that the proper transform $\Psi_{\mathfrak{R}}^{*}H$ of $H$ by
the morphism $\Psi_{\mathfrak{R}}:\mathfrak{R}\to\mathbb{P}(W)$ with
image $\mathcal{S}$ is the zero locus of a section belonging to the
sub-$\mathrm{Aut}(\mathfrak{C})$-module $W$ of $H^{0}(\mathfrak{R},\mathcal{O}_{\mathfrak{R}}(5f_{1}+f_{2}))$
and whose stabilizer contains the subgroup $c_{\mathfrak{C}}(c_{\mathcal{C}}^{-1}(G_{p}))$.
The latter being a Borel subgroup $B=\mathbb{G}_{a}\rtimes\mathbb{G}_{m}=\mathrm{Spec}(k[a,\lambda^{\pm1}])$
of $\mathrm{Aut}(\mathfrak{C})\cong\mathrm{PGL}_{2}$, we can choose
coordinates $[x:y]$ on $\mathfrak{C}\cong\mathbb{P}^{1}$ for which
$B$ acts by $[x:y]\mapsto[\lambda x:y+ax]$. Since $W$ is an irreducible
$\mathrm{Aut}(\mathfrak{C})$-module and $B$ is a Borel subgroup
of $\mathrm{Aut}(\mathfrak{C}),$ $W$ contains a unique $B$-stable
line which is readily checked from the explicit basis (\ref{eq:basis-FN})
of $W$ to be equal to that generated by the element $e_{0}=x_{1}^{5}x_{2}$.
So, $\Psi_{\mathfrak{R}}^{*}H=Z(e_{0})$, and hence, by Proposition
\ref{prop:Hilb-Evaluation-Morphism} b), $F_{p}\cap\mathcal{S}=5\ell_{p}\cup\Gamma_{p}$.
The equality $\mathrm{Aut}(F_{p})=\mathrm{Aut}(F_{p},\ell_{p},p)$
and the injectivity of the restriction homomorphism $\mathrm{Aut}(V_{5},F_{p})\to\mathrm{Aut}(F_{p})$
are clear from the description of $H$ and the induced action of $G_{p}$
on $W/\langle e_{0}\rangle$. 
\end{proof}
We now recollect other basic properties of the non-normal hyperplane
section $F_{p}$ of $V_{5}\subset\mathbb{P}(W)$. 
\begin{prop}
\label{prop:Normalization-Special-Hyperplane-Section} Let $\pi:\mathfrak{F}_{[\ell_{p}]}:=\mathbb{P}(\mathcal{E}|_{T_{[\ell_{p}]}\mathfrak{C}})\to T_{[\ell_{p}]}\mathfrak{C}$
be the restriction of $\pi:\mathfrak{U}\to\mathfrak{H}$ over the
projective tangent line $T_{[\ell_{p}]}\mathfrak{C}$ to the conic
$\mathfrak{C}\subset\mathfrak{H}$ at the point $[\ell_{p}]$ corresponding
to the line $\ell_{p}$. Then the following hold:

a) The projective bundle $\pi:\mathfrak{F}_{[\ell_{p}]}\to T_{[\ell_{p}]}\mathfrak{C}$
is isomorphic to the Hirzebruch surface $\rho_{3}:\mathbb{F}_{3}\to\mathbb{P}^{1}$
and $\mathfrak{F}_{[\ell_{p}]}\cap j(\mathfrak{R})$ is a $2$-section
of $\pi$ which decomposes as the union of the exceptional section
$s_{[\ell_{p}]}$ of $\pi:\mathfrak{F}_{[\ell_{p}]}\to T_{[\ell_{p}]}\mathfrak{C}$
with self-intersection number $-3$ and of a section $\Upsilon_{p}\sim s_{[\ell_{p}]}+4f_{[\ell_{p}]}$,
where $f_{[\ell_{p}]}=\pi^{-1}([\ell_{p}])$, meeting $s_{[\ell_{p}]}$
transversally at the unique point 
\[
\mathfrak{p}=f_{[\ell_{p}]}\cap j(\Delta(\mathfrak{C}))=f_{[\ell_{p}]}\cap s_{[\ell_{p}]}.
\]

b) The morphism $\Psi|_{\mathfrak{F}_{[\ell_{p}]}}:\mathfrak{F}_{[\ell_{p}]}\to\mathbb{P}(W)$
factors through $F_{p}$ and the induced morphism $\nu:\mathfrak{F}_{[\ell_{p}]}\to F_{p}$
is the normalization of $F_{p}$. Moreover $\nu^{*}\omega_{F_{p}}^{\vee}\cong\mathcal{O}_{\mathfrak{F}_{[\ell_{p}]}}(s_{[\ell_{p}]}+4f_{[\ell_{p}]})\cong\omega_{\mathcal{\mathfrak{F}}_{[\ell_{p}]}}^{\vee}\otimes\mathcal{O}_{\mathfrak{F}_{[\ell_{p}]}}(-s_{[\ell_{p}]}-f_{[\ell_{p}]})$. 

c) The inverse image by $\nu$ of every line in $F_{p}$ other than
$\ell_{p}$ is the union of a fiber $f$ of $\pi:\mathfrak{F}_{[\ell_{p}]}\to T_{[\ell_{p}]}\mathfrak{C}$
other than $f_{[\ell_{p}]}$ and of the unique point $q\in f_{[\ell_{p}]}$
such that $\nu(q)=\nu(f\cap s_{[\ell_{p}]})$. On the other hand,
$\nu^{-1}\ell_{p}=s_{[\ell_{p}]}\cup f_{[\ell_{p}]}$. 

d) The morphism $\nu$ restricts to an isomorphism $\mathfrak{F}_{[\ell_{p}]}\setminus(s_{[\ell_{p}]}\cup f_{[\ell_{p}]})\to F_{p}\setminus\ell_{p}$,
maps the pairs $(s_{[\ell_{p}]},\mathfrak{p})$ and $(f_{[\ell_{p}]},\mathfrak{p})$
isomorphically onto the pair $(\ell_{p},p)$ and maps the pair $(\Upsilon_{p},\mathfrak{p})$
isomorphically onto the pair $(\Gamma_{p},p)$. 
\end{prop}

\begin{proof}
Since $T_{[\ell_{p}]}\mathfrak{C}$ is the tangent line to the branch
curve $\mathfrak{C}$ of $\upsilon:\mathfrak{R}\to\mathfrak{H}$,
\cite[Proposition 8]{Sch61} implies that 
\[
\mathcal{E}|_{T_{[\ell_{p}]}\mathfrak{C}}=\upsilon_{*}(\mathrm{p}_{1}^{*}(\omega_{\mathfrak{C}}^{\vee})^{\otimes2})|_{T_{[\ell_{p}]}\mathfrak{C}}\cong\upsilon_{*}\mathcal{O}_{\mathfrak{R}}(4f_{1})|_{T_{[\ell_{p}]}\mathfrak{C}}\cong\mathcal{O}_{T_{[\ell_{p}]}\mathfrak{C}}(3)\oplus\mathcal{O}_{T_{[\ell_{p}]}\mathfrak{C}},
\]
which give the isomorphism $(\mathfrak{F}_{[\ell_{p}]},(\mathcal{O}_{\mathbb{P}(\mathcal{E})}(1)\otimes\pi^{*}\mathcal{O}_{\mathfrak{H}}(1))|_{\mathfrak{F}_{[\ell_{p}]}})\cong(\mathbb{F}_{3},\mathcal{O}_{\mathbb{F}_{3}}(s_{[\ell_{p}]}+4f_{[\ell_{p}]}))$.
The curve $\upsilon^{-1}(T_{|\ell_{p}]}\mathfrak{C})$ is the union
of the fibers $f_{1}$ and $f_{2}$ of the first and second projections
$\mathrm{p}_{1},\mathrm{p}_{2}:\mathfrak{R}\to\mathfrak{C}$ passing
through the point $\upsilon^{-1}([\ell_{p}])\in\Delta(\mathfrak{C})$.
By construction of the closed immersion of $\mathfrak{H}$-schemes
$j:\mathfrak{R}\to\mathfrak{U}=\mathbb{P}(\mathcal{E})$ in Proposition
\ref{prop:Hilb-Evaluation-Morphism} a), we have $\mathfrak{F}_{[\ell_{p}]}\cap j(\mathfrak{R})=j(f_{1})\cup j(f_{2})$
and $j^{*}(\mathcal{O}_{\mathbb{P}(\mathcal{E})}(1)\otimes\pi^{*}\mathcal{O}_{\mathfrak{H}}(1))\cong\mathcal{O}_{\mathfrak{R}}(5f_{1}+f_{2})$.
In combination with the previous isomorphism, this implies that the
intersection numbers of $j(f_{1})$ and $j(f_{2})$ with $s_{[\ell_{p}]}+4f_{[\ell_{p}]}$
are equal to $1$ and $5$ respectively, and a straightforward computation
then shows that $j(f_{1})=s_{[\ell_{p}]}$ and that $j(f_{2})\sim s_{[\ell_{p}]}+4f_{[\ell_{p}]}$,
proving assertion a).

Since $(s_{[\ell_{p}]}+4f_{[\ell_{p}]})^{2}=5$ and the morphism $\Psi:\mathfrak{U}\to\mathbb{P}(W)$
is finite of degree $3$ onto its image and, by Proposition \ref{prop:Hilb-Evaluation-Morphism}
d), \'etale outside $\mathfrak{U}|_{\mathfrak{C}}\cup j(\mathfrak{R})$,
$\Psi(\mathfrak{F}_{[\ell_{p}]})$ is an irreducible and reduced surface
$F$ contained in $V_{5}$, whose degree in $\mathbb{P}(W)$ is either
$1$ or $5$. The first case is excluded since $\Psi|_{\mathfrak{F}_{[\ell_{p}]}}:\mathfrak{F}_{[\ell_{p}]}\to\Psi(\mathfrak{F}_{[\ell_{p}]})$
is finite of degree at most $3$. Thus $F$ is a hyperplane section
of $V_{5}$, the induced morphism $\nu:\mathfrak{F}_{[\ell_{p}]}\to F$
is finite and birational, whence equals the normalization of $F$.
Moreover, by Proposition \ref{prop:Hilb-Evaluation-Morphism} a) and
b), we have $\nu(s_{[\ell_{p}]})=\psi_{\mathfrak{R}}(f_{1})=\ell_{p}=\nu(f_{[\ell_{P}]})$
which implies in particular that $F$ is singular along $\ell_{p}$,
whence equal to $F_{p}$, and that $\nu(\Upsilon_{p})=\psi_{\mathfrak{R}}(f_{2})=\Gamma_{p}$,
which proves assertion d). Since $F_{p}$ is a Cartier divisor on
$V_{5}$, $F_{p}$ is a Gorenstein scheme whose dualizing sheaf $\omega_{F_{p}}\cong\mathcal{O}_{\mathbb{P}(W)}(-1)|_{F_{p}}$
is the canonical sheaf of $F_{p}$. On the other hand, since $\Psi^{*}\mathcal{O}_{\mathbb{P}(W)}(1)\cong\mathcal{O}_{\mathbb{P}(\mathcal{E})}(1)\otimes\pi^{*}\mathcal{O}_{\mathfrak{H}}(1)$,
we have 
\begin{align*}
\nu^{*}\omega_{F_{p}}^{\vee} & \cong\Psi|_{\mathfrak{F}_{[\ell_{p}]}}^{*}\mathcal{O}_{\mathbb{P}(W)}(1)|_{F_{p}}\cong(\mathcal{O}_{\mathbb{P}(\mathcal{E})}(1)\otimes\pi^{*}\mathcal{O}_{\mathfrak{H}}(1))|_{\mathfrak{F}_{[\ell_{p}]}}\cong\mathcal{O}_{\mathfrak{F}_{[\ell_{p}]}}(s_{[\ell_{p}]}+4f_{[\ell_{p}]})\\
 & \cong\mathcal{O}_{\mathfrak{F}_{[\ell_{p}]}}(s_{[\ell_{p}]}+4f_{[\ell_{p}]})\cong\omega_{\mathcal{\mathfrak{F}}_{[\ell_{p}]}}^{\vee}\otimes\mathcal{O}_{\mathfrak{F}_{[\ell_{p}]}}(-s_{[\ell_{p}]}-f_{[\ell_{p}]})
\end{align*}
which gives assertion b). Assertion c) is immediate. 
\end{proof}
\begin{rem}
\label{rem:equalizer}With the notation of Proposition \ref{prop:Normalization-Special-Hyperplane-Section},
let $\nu|_{s_{[\ell_{p}]}}:s_{[\ell_{p}]}\to\ell_{p}$ and $\nu|_{f_{[\ell_{p}]}}:f_{[\ell_{p}]}\to\ell_{p}$
be the isomorphisms induced by $\nu:\mathfrak{F}_{[\ell_{p}]}\to F_{p}$
and let $\mathcal{I}_{p}\subset\mathcal{O}_{F_{p}}$ and $\mathcal{I}_{\mathfrak{p}}\subset\mathcal{O}_{\mathfrak{F}_{[\ell_{p}]}}$
be the ideal sheaves of the points $p\in F_{p}$ and $\mathfrak{p}\in\mathfrak{F}_{[\ell_{p}]}$.
Then, see e.g. \cite[Theorem 2.6]{Re94}, $\nu^{*}H^{0}(F_{p},\omega_{F_{p}}^{\vee})\subset H^{0}(\mathfrak{F}_{[\ell_{p}]},\nu^{*}\omega_{F_{p}}^{\vee})$
is the equalizer of the two maps 
\begin{align*}
H^{0}(\mathfrak{F}_{[\ell_{p}]},\nu^{*}\omega_{F_{p}}^{\vee}) & \to H^{0}(s_{[\ell_{p}]},\nu^{*}\omega_{F_{p}}^{\vee}|_{s_{[\ell_{p}]}})\cong H^{0}(s_{[\ell_{p}]},\omega_{s_{[\ell_{p}]}}^{\vee}\otimes\mathcal{I}_{\mathfrak{p}}|_{s_{[\ell_{p}]}})\cong H^{0}(s_{[\ell_{p}]},(\nu|_{s_{[\ell_{p}]}})^{*}(\omega_{\ell_{p}}^{\vee}\otimes\mathcal{I}_{p}|_{\ell_{p}}))\\
H^{0}(\mathfrak{F}_{[\ell_{p}]},\nu^{*}\omega_{F_{p}}^{\vee}) & \to H^{0}(f_{[\ell_{p}]},\nu^{*}\omega_{F_{p}}^{\vee}|_{f_{[\ell_{p}]}})\cong H^{0}(f_{[\ell_{p}]},\omega_{f_{[\ell_{p}]}}^{\vee}\otimes\mathcal{I}_{\mathfrak{p}}|_{f_{[\ell_{p}]}})\cong H^{0}(f_{[\ell_{p}]},(\nu|_{f_{[\ell_{p}]}})^{*}(\omega_{\ell_{p}}^{\vee}\otimes\mathcal{I}_{p}|_{\ell_{p}})).
\end{align*}

By Proposition \ref{prop:Normalization-Special-Hyperplane-Section}
d), $\Omega_{(s_{[\ell_{p}]}\cup f_{[\ell_{p}]})/\ell_{p}}|_{\mathfrak{p}}\cong\Omega_{\mathfrak{F}_{[\ell_{p}]/F_{p}}}|_{\mathfrak{p}}$
is a $1$-dimensional $k$-vector space. Moreover, since $s_{[\ell_{p}]}$
and $f_{[\ell_{p}]}$ intersect transversally at $\mathfrak{p}$,
$\Omega_{\mathfrak{F}_{[\ell_{p}]}}|_{\mathfrak{p}}$ is the direct
sum of the restrictions to $\mathfrak{p}$ of the conormal sheaves
$\mathcal{C}_{s_{[\ell_{p}]}/\mathfrak{F}_{[\ell_{p}]}}$ and $\mathcal{C}_{f_{[\ell_{p}]}/\mathfrak{F}_{[\ell_{p}]}}$
and the induced projections from the two factors to $\Omega_{(s_{[\ell_{p}]}\cup f_{[\ell_{p}]})/\ell_{p}}|_{\mathfrak{p}}$
are surjective. The surjections $\Omega_{\mathfrak{F}_{[\ell_{p}]}}|_{\mathfrak{p}}\to\mathcal{C}_{s_{[\ell_{p}]}/\mathfrak{F}_{[\ell_{p}]}}|_{\mathfrak{p}}$,
$\Omega_{\mathfrak{F}_{[\ell_{p}]}}|_{\mathfrak{p}}\to\mathcal{C}_{f_{[\ell_{p}]}/\mathfrak{F}_{[\ell_{p}]}}|_{\mathfrak{p}}$
and $\Omega_{\mathfrak{F}_{[\ell_{p}]}}|_{\mathfrak{p}}\to\Omega_{(s_{[\ell_{p}]}\cup f_{[\ell_{p}]})/\ell_{p}}|_{\mathfrak{p}}$
determine three distinct points
\begin{equation}
q_{s_{[\ell_{p}]}},q_{f_{[\ell_{p}]}},\delta_{[\ell_{p}]}\in\mathbb{P}(\Omega_{\mathfrak{F}_{[\ell_{p}]}}|_{\mathfrak{p}})\label{eq:Def_delta_p}
\end{equation}
and we have the following characterization:.
\end{rem}

\begin{lem}
\label{lem:Smooth-Criterion-proper-transform}Let $\alpha_{\mathfrak{p}}:\hat{\mathfrak{F}}_{[\ell_{p}]}\to\mathfrak{F}_{[\ell_{p}]}$
be the blow-up of $\mathfrak{F}_{[\ell_{p}]}$ at $\mathfrak{p}$
with exceptional divisor $\mathfrak{E}_{[\ell_{p}]}\cong\mathbb{P}(\Omega_{\mathfrak{F}_{[\ell_{p}]}}|_{\mathfrak{p}})$.
Then the image by $\nu:\mathfrak{F}_{[\ell_{p}]}\to F_{p}$ of a curve
$C\subset\mathfrak{F}_{[\ell_{p}]}$ passing through $\mathfrak{p}$
and non-singular at $\mathfrak{p}$ is singular at $p$ if and only
if its proper transform in $\hat{\mathfrak{F}}_{[\ell_{p}]}$ intersects
$\mathfrak{E}_{\mathfrak{[\ell_{p}]}}$ at $\delta_{[\ell_{p}]}$.
\end{lem}

\begin{proof}
Indeed, the image $\nu_{*}C$ of $C$ is singular at $p$ if and only
if the restriction $\mathcal{C}_{C/\mathfrak{F}_{[\ell_{p}]}}|_{\mathfrak{p}}$
at $\mathfrak{p}$ of its conormal sheaf $\mathcal{C}_{C/\mathfrak{F}_{[\ell_{p}]}}\subset\Omega_{\mathfrak{F}_{[\ell_{p}]}}|_{C}$
belongs to the kernel of the surjection $\Omega_{\mathfrak{F}_{[\ell_{p}]}}|_{\mathfrak{p}}\to\Omega_{\mathfrak{F}_{[\ell_{p}]}/F_{p}}|_{\mathfrak{p}}\cong\Omega_{s_{[\ell_{p}]}\cup f_{[\ell_{p}]}/\ell_{p}}|_{\mathfrak{p}}$. 
\end{proof}
We also record the following related result that will be used later
on in the proof of Proposition \ref{prop:Hilbert-scheme-correspondence}
in Section \ref{sec:X_22-stuff}. 
\begin{lem}
\label{lem:Propoer-transforms-special-quintics} Let $\beta_{\ell_{p}}:\hat{V}_{5}\to V_{5}$
be the blow-up with center at $\ell_{p}$ with exceptional divisor
$\rho:E_{\ell_{p}}=\mathbb{P}(\mathcal{C}_{\ell_{p}/V_{5}})\to\ell_{p}$
and let $s_{2}$ be the exceptional section of $\rho$ with self-intersection
number $-2$. Then the proper transforms $\hat{\Gamma}_{p'}$ in $\hat{V}_{5}$
of the special rational normal quintic curve $\Gamma_{p'}\subset\mathcal{S}$
for all $p'\in\mathcal{C}$ intersect $E_{\ell_{p}}$ at pairwise
distinct points of $s_{2}$.
\end{lem}

\begin{proof}
Since $\ell_{p}\subset V_{5}$ is a line of special type, its conormal
sheaf $\mathcal{C}_{\ell_{p}/V_{5}}$ is isomorphic to $\mathcal{O}_{\ell_{p}}(-1)\oplus\mathcal{O}_{\ell_{p}}(1)$,
so that $E_{\ell_{p}}\cong\mathbb{F}_{2}$ with exceptional section
$s_{2}$ corresponding to the first projection $\mathcal{O}_{\ell_{p}}(-1)\oplus\mathcal{O}_{\ell_{p}}(1)\to\mathcal{O}_{\ell_{p}}(-1)$.
The line $\ell_{p}\subset\mathcal{S}$ is the image of a certain fiber
$f_{1}\subset\mathfrak{R}$ of the first projection $\mathrm{p}_{1}:\mathfrak{R}=\mathfrak{C\times}\mathfrak{C}\to\mathfrak{C}$
by the morphism $\Psi_{\mathfrak{R}}=\Psi\circ j:\mathfrak{R}\to\mathfrak{U}\to\mathbb{P}(W)$
with image $\mathcal{S}$ of Proposition \ref{prop:Hilb-Evaluation-Morphism}
a). The morphism $\Psi_{\mathfrak{R}}$ thus induces a homomorphism
of $\mathcal{O}_{f_{1}}$-modules $(\Psi_{\mathfrak{R}}|_{f_{1}})^{*}\mathcal{C}_{\ell_{p}/V_{5}}\cong\mathcal{O}_{f_{1}}(-1)\oplus\mathcal{O}_{f_{1}}(1)\to\mathcal{C}_{f_{1}/\mathfrak{R}}\cong\mathcal{O}_{f_{1}}$
which, since $\mathrm{Hom}(\mathcal{O}_{f_{1}}(1),\mathcal{O}_{f_{1}})=0$,
factors through the projection to $\mathcal{O}_{\ell_{p}}(-1)$. This
implies that the proper transform $\hat{\mathcal{S}}$ of $\mathcal{S}$
in $\hat{V}_{5}$ intersects $E_{\ell_{p}}$ along $s_{2}$. A special
rational normal quintic curve $\Gamma_{p'}\subset\mathcal{S}$, $p'\in\mathcal{C}\setminus\{p\}$,
being by definition the image by $\psi_{\mathfrak{R}}:\mathfrak{R}\to\mathcal{S}$
of a fiber $f_{2}$ of the second projection $\mathrm{p}_{2}:\mathfrak{R}=\mathfrak{C\times}\mathfrak{C}\to\mathfrak{C}$
other than that meeting $f_{1}$ at the point $f_{1}\cap\Delta(\mathfrak{C})$,
it follows that the proper transforms $\hat{\Gamma}_{p'}\subset\hat{\mathcal{S}}$,
$p'\in\mathcal{C}\setminus\{p\}$, intersect $E_{\ell_{p}}$ at pairwise
distinct points of $s_{2}=\hat{\mathrm{\mathcal{S}}}\cap E_{\ell_{p}}$. 
\end{proof}

\subsection{\protect\label{subsec:Classification-Quintic-Curves-Infinite-Stabilizers}
Classification of rational normal quintic curves with infinite stabilizers }

We now classify rational normal quintic curves $\Gamma\subset V_{5}$
with infinite stabilizers $G_{\Gamma}=\mathrm{Aut}(V_{5},\Gamma)\subset\mathrm{Aut}(V_{5})$,
up to the action of $\mathrm{Aut}(V_{5})$. 

\subsubsection{Preliminaries}

Recall that a rational normal quintic curve $\Gamma\subset V_{5}\subset\mathbb{P}(W)$
is contained in a unique hyperplane $H$ of $\mathbb{P}(W)$, hence
in a unique hyperplane section $F=V_{5}\cap H$ of $V_{5}$. We begin
with the following lemma which identifies hyperplane sections of $V_{5}$
that can contain rational normal quintic curves $\Gamma$ with infinite
stabilizer. 
\begin{lem}
\label{prop:RationalQuintic-inifnite-stabilizer} Let $\Gamma\subset V_{5}$
be a rational normal quintic curve with infinite stabilizer $G_{\Gamma}=\mathrm{Aut}(V_{5},\Gamma)$.
Then the unique hyperplane section of $V_{5}$ containing $\Gamma$
is a hyperplane section $F_{p}$ passing through a $G_{\Gamma}$-stable
point $p$ of $\mathcal{C}$ and singular along the unique line of
special type $\ell_{p}$ passing through $p$. 
\end{lem}

\begin{proof}
The stabilizer $G_{F}=\mathrm{Aut}(V_{5},F)$ in $\mathrm{Aut}(V_{5})$
of the unique hyperplane section $F=V_{5}\cap H$ containing $\Gamma$
is a proper subgroup of $\mathrm{Aut}(V_{5})$ containing $G_{\Gamma}$.
By definition of $\Psi_{\mathfrak{R}}:\mathfrak{R}\to\mathbb{P}(W)$,
$\Psi_{\mathfrak{R}}^{*}H$ is the zero locus on $\mathfrak{R}$ of
a section belonging to the sub-$\mathrm{Aut}(\mathfrak{C})$-module
$W$ of $H^{0}(\mathfrak{R},\mathcal{O}_{\mathfrak{R}}(5f_{1}+f_{2}))$
and whose stabilizer contains the infinite proper connected subgroup
$G_{F}^{0}$. If $G_{F}^{0}$ is not a torus, then $G_{F}^{0}$ is
either a Borel subgroup $B\cong\mathbb{G}_{a}\rtimes\mathbb{G}_{m}=\mathrm{Spec}(k[a,\lambda^{\pm1}])$
of $\mathrm{Aut}(\mathfrak{C})\cong\mathrm{PGL}_{2}$ or its unipotent
radical. In both cases, $H$ intersects $\mathcal{C}$ with multiplicity
$6$ at the unique fixed point $p$ for the induced action of $G_{F}^{0}$.
The restriction homomorphism $H^{0}(\mathbb{P}(W),\mathcal{O}_{\mathbb{P}(W)}(1))\to H^{0}(\mathcal{C},\mathcal{O}_{\mathcal{C}}(6))$
being by construction a $(\mathrm{PGL}_{2}$-equivariant) isomorphism,
there exists a unique $H$ intersecting $\mathcal{C}$ with multiplicity
$6$ at $p$. So $F$ is unique, hence $G_{p}$-stable, and consequently
equal to $F_{p}$ by Lemma \ref{lem:non-normal-section-charac}. 

Now assume that $G_{F}^{0}$ is a torus. Then $G_{\Gamma}^{0}=G_{F}^{0}$
and choosing coordinates $[x:y]$ on $\mathfrak{C}\cong\mathbb{P}^{1}$
so that the action of $G_{F}^{0}\cong\mathbb{G}_{m}=\mathrm{Spec}(k[\lambda^{\pm1}])$
on $\mathfrak{C}\cong\mathbb{P}^{1}$ is given by $[x:y]\mapsto[\lambda x:y]$,
we have $\Psi_{\mathfrak{R}}^{*}H=Z(s)$ for some $s\in W$ generating
a $\mathbb{G}_{m}$-stable $k$-vector subspace. Since the sections
$e_{i}$, $i=0,\ldots,6$, in (\ref{eq:basis-FN}) form a basis of
the weight space decomposition of $W$ with respect to the action
of $\mathbb{G}_{m}$, with respective weights $w_{i}=6-i$, we have
$\Psi_{\mathfrak{R}}^{*}H=Z(e_{i})$ for a certain $i\in\{0,\ldots,6\}$.
This implies in turn that the weights of $H^{0}(\Gamma,\mathcal{O}_{\mathbb{P}^{6}}(1)|_{\Gamma})$
for the induced faithful action of $G_{\Gamma}^{0}\subset\mathrm{Aut}(\Gamma)$
with respect to induced $G_{\Gamma}^{0}$-linearization of $\mathcal{O}_{\mathbb{P}^{6}}(1)|_{\Gamma}\cong\mathcal{O}_{\Gamma}(5)$
are equal to the $w_{j}$, $j\neq i$. Since on the other hand for
any choice of $G_{\Gamma}^{0}$-linearization of $\mathcal{O}_{\Gamma}(1)$
the corresponding weights of $H^{0}(\Gamma,\mathcal{O}_{\Gamma}(5))$
for the induced $G_{\Gamma}^{0}$-linearization of $\mathcal{O}_{\Gamma}(5)$
form a sequence of $6$ consecutive integers, the only possibility
is that $i=0$ or $6$, that is, $\Psi_{\mathfrak{R}}^{*}H=Z(x_{1}^{5}y_{2})$
or $Z(y_{1}^{5}x_{2})$. We conclude as in the previous case that
$F$ is one of the non-normal hyperplane sections $F_{p}$ or $F_{p'}$,
where $p$ and $p'$ are the two fixed points of the action of $c_{\mathcal{C}}(G_{F}^{0})\subset\mathrm{Aut}(\mathcal{C})$. 
\end{proof}
The stabilizer $G_{\Gamma}$ of a rational normal quintic curve $\Gamma\subset V_{5}$
contained in the hyperplane section $F_{p}$ of $V_{5}$ singular
along the unique line $\ell_{p}$ of special type passing through
a point $p\in\mathcal{C}$ is a subgroup of the Borel subgroup $G_{p}=c_{\mathcal{C}}^{-1}(\mathrm{Aut}(\mathcal{C},p))$
of $\mathrm{Aut}(V_{5})$. The group $G_{p}$ also identifies with
the Borel subgroup $c_{\mathfrak{C}}(G_{p})=\mathrm{Aut}(\mathfrak{C},[\ell_{p}])$
of $\mathrm{Aut}(\mathfrak{C})$ and, with the notation of Proposition
\ref{prop:Normalization-Special-Hyperplane-Section} and Lemma \ref{lem:Smooth-Criterion-proper-transform},
we have the following:
\begin{lem}
\label{lem:inverse-image-quintic-normalization}\label{lem:Action-on-special-hyperplane-sections}The
normalization morphism $\nu:\mathfrak{F}_{[\ell_{p}]}\to F_{p}$ is
$G_{p}$-equivariant and the following hold: 

a) The induced action of $G_{p}$ on $\mathfrak{F}_{[\ell_{p}]}$
is faithful with $\mathfrak{p}$ as a unique fixed point.

b) The curves $s_{[\ell_{p}]}$, $f_{[\ell_{p}]}$ and $\Upsilon_{p}$
are $G_{p}$-stable, and the restriction homomorphisms $G_{p}\to\mathrm{Aut}(s_{[\ell_{p}]},\mathfrak{p})$,
$G_{p}\to\mathrm{Aut}(f_{[\ell_{p}]},\mathfrak{p})$ and $G_{p}\to\mathrm{Aut}(\Upsilon_{p},\mathfrak{p})$
are isomorphisms. 

c) The proper transform by $\nu$ of a rational normal quintic curve
$\Gamma\subset F_{p}$ is an integral member of the complete linear
system $|s_{[\ell_{p}]}+4f_{[\ell_{p}]}|$. Moreover $\nu$ induces
one-to-one stabilizer preserving correspondence between rational normal
quintic curves $\Gamma\subset F_{p}$ passing through $p$ and integral
members $\Upsilon$ of $|s_{[\ell]}+4f_{[\ell]}|$ passing through
$\mathfrak{p}$ and whose proper transforms in the blow-up $\alpha_{\mathfrak{p}}:\hat{\mathfrak{F}}_{[\ell_{p}]}\to\mathfrak{F}_{[\ell_{p}]}$
at $\mathfrak{p}$ intersect the exceptional divisor $\mathfrak{E}_{[\ell_{p}]}$
away from the point $\delta_{[\ell_{p}]}$. 
\end{lem}

\begin{proof}
The fact that $\nu=\psi|_{\mathfrak{F}_{[\ell_{p}]}}:\mathfrak{F}_{[\ell_{p}]}\to F_{p}$
is $G_{p}$-equivariant follows from the construction of $\psi:\mathfrak{U}\to V_{5}$
and the canonical identifications made in (\ref{eq:canonical-ID-Aut(V5)}).
With the notation of the proof of Proposition \ref{prop:Normalization-Special-Hyperplane-Section},
the group $G_{p}$ identifies further with the stabilizer of the point
$\upsilon^{-1}([\ell_{p}])\in\Delta(\mathfrak{C})$ for the diagonal
action of $\mathrm{Aut}(\mathfrak{C})$ on $\mathfrak{R}$. The fact
that the curves $s_{[\ell_{p}]}$, $f_{[\ell_{p}]}$ and $\Upsilon_{p}$
are $G_{p}$-stable and that the corresponding restriction homomorphisms
are isomorphisms then follows from the identifications $j(f_{1})=s_{[\ell_{p}]}$
and $j(f_{2})=\Upsilon_{p}$ and Proposition \ref{prop:Normalization-Special-Hyperplane-Section}
d). This also implies that the induced actions of $G_{p}$ on $\mathfrak{F}_{[\ell_{p}]}$
and $T_{[\ell_{p}]}\mathfrak{C}$ are faithful, and that the restriction
homomorphism $G_{p}\to\mathrm{Aut}(T_{[\ell_{p}]}\mathfrak{C},[\ell_{p}])$
is an isomorphism, from which it follows in turn that $\mathfrak{p}$
is the unique $G_{p}$-fixed point on $\mathfrak{F}_{[\ell_{p}]}$.
This proves assertions a) and b).

The proper transform $\Upsilon=\nu_{*}^{-1}\Gamma$ in $\mathfrak{F}_{[\ell_{p}]}\cong\mathbb{F}_{3}$
of a rational normal quintic curve $\Gamma\subset F_{p}$ is a non-singular
rational curve linearly equivalent to $as_{[\ell_{p}]}+bf_{[\ell_{p}]}$
for some non-negative integers $a,b$. By the projection formula for
$\nu$ combined with the isomorphism $\nu^{*}\omega_{F_{p}}^{\vee}\cong\mathcal{O}_{\mathfrak{F}_{[\ell_{p}]}}(s_{[\ell_{p}]}+4f_{[\ell_{p}]})$
of Proposition \ref{prop:Normalization-Special-Hyperplane-Section}
b), we have $a+b=\Upsilon\cdot(s_{[\ell_{p}]}+4f_{[\ell_{p}]})=\deg(\omega_{F_{p}}^{\vee}|_{\Gamma})=5$.
Since the curve $\Upsilon$ is irreducible, the only possibility is
$(a,b)=(1,4)$. Conversely, Lemma \ref{lem:Smooth-Criterion-proper-transform}
implies that the image $\Gamma=\nu_{*}\Upsilon$ of a non-singular
member $\Upsilon$ of the complete linear system $|s_{[\ell_{p}]}+4f_{[\ell_{p}]}|$
passing through $\mathfrak{p}$ is non-singular, whence a rational
normal quintic curve, if and only if $(\alpha_{\mathfrak{p}}^{-1})_{*}\Upsilon\cap\mathfrak{E}_{[\ell_{p}]}\neq\delta_{[\ell_{p}]}$.
Finally, the universal property of the normalization morphism $\nu:\mathfrak{F}_{[\ell_{p}]}\to F_{p}$
provides an isomorphism of groups $\nu^{*}:\mathrm{Aut}(F_{p})=\mathrm{Aut}(F_{p},\ell_{p},p)\to\mathrm{Aut}(\mathfrak{F}_{[\ell_{p}]},s_{[\ell_{p}]},f_{[\ell_{p}]},\mathfrak{p})$
which, because the induced $G_{p}$-action on $\mathfrak{F}_{[\ell_{p}]}$
is faithful maps $G_{p}\subset\mathrm{Aut}(F_{p})$ isomorphically
onto its image. This implies in turn that $\nu^{*}$ maps the stabilizer
$G_{\Gamma}\subset G_{p}\subset\mathrm{Aut}(F_{p},\ell_{p},p)$ of
a rational normal quintic curve $\Gamma$ passing through $p$ isomorphically
onto the stabilizer $G_{\Upsilon}\subset\mathrm{Aut}(\mathfrak{F}_{[\ell_{p}]},\mathfrak{p})=\mathrm{Aut}(\mathfrak{F}_{[\ell_{p}]},s_{[\ell_{p}]},f_{[\ell_{p}]},\mathfrak{p})$
of it proper transform $\Upsilon=\nu_{*}^{-1}\Gamma$. This completes
the proof of assertion c). 
\end{proof}

\subsubsection{\protect\label{subsec:Classification}The Classification }

Lemma \ref{lem:inverse-image-quintic-normalization} implies that
for a rational normal quintic curve $\Gamma\subset F_{p}$ with infinite
stabilizer $G_{\Gamma}$, the stabilizer $G_{\Upsilon}=G_{\Gamma}$
of the curve $\Upsilon=\nu_{*}^{-1}\Gamma$ is a subgroup of $G_{p}$,
whose neutral component $G_{\Upsilon}^{0}$ is thus either equal to
$G_{p}$ or to its the unipotent radical $G_{p}^{a}\cong\mathbb{G}_{a}$
or to a maximal torus $T\cong\mathbb{G}_{m}$ of $G_{p}$. We now
divide the classification of rational normal quintic curves $\Gamma=\nu_{*}\Upsilon$
into two parts according to whether $G_{\Upsilon}=G_{\Gamma}$ contains
or not a maximal torus of $G_{p}$. The various families of curves
which will be introduced below are summarized on an explicit model
in Example \ref{exa:Explicit-model-Take-1}. We henceforth use the
shorter notation $\ell$ for the line $\ell_{p}$. 

\medskip

$\bullet$ Case 1: $G_{\Gamma}$ contains a maximal torus $T\subset G_{p}$.
Since the restriction homomorphisms $G_{p}\to\mathrm{Aut}(s_{[\ell]},\mathfrak{p})$
and $G_{p}\to\mathrm{Aut}(f_{[\ell]},\mathfrak{p})$ are isomorphisms
by Lemma \ref{lem:Action-on-special-hyperplane-sections} a), the
inverse images by $\nu:\mathfrak{F}_{[\ell]}\to F_{p}$ of the unique
$T$-fixed point on $\ell\setminus\{p\}$ form a unique pair of $T$-fixed
points $r_{T}\sqcup r{}_{T}'\in(s_{[\ell]}\setminus\{\mathfrak{p}\})\sqcup(f_{[\ell]}\setminus\{\mathfrak{p}\})$.
The intersection $r_{T}''$ of the fiber $f_{T}$ of $\pi:\mathfrak{F}_{[\ell]}\to T_{[\ell]}\mathfrak{C}$
passing through $r_{T}$ with $\Upsilon_{p}$ is then a $T$-fixed
point. Since the restriction homomorphism 
\[
\mathrm{res}|_{\Upsilon_{p}}:H^{0}(\mathfrak{F}_{[\ell]},\mathcal{O}_{\mathfrak{F}_{[\ell]}}(s_{[\ell]}+3f_{[\ell]}))\to H^{0}(\Upsilon_{p},\mathcal{O}_{\mathfrak{F}_{[\ell]}}(s_{[\ell]}+3f_{[\ell]})|_{\Upsilon_{p}})\cong H^{0}(\mathbb{P}^{1},\mathcal{O}_{\mathbb{P}^{1}}(4))
\]
is an isomorphism, there exists a unique section $s_{T}\sim s_{[\ell]}+3f_{[\ell]}$
of $\pi$ intersecting $\Upsilon_{p}$ at $r_{T}''$ with multiplicity
$4$. The curve $s_{T}$ is necessarily $T$-stable, and since it
is disjoint from $s_{[\ell]}$, it intersects $f_{[\ell]}$ at the
$T$-fixed point $r_{T}'$. 
\begin{prop}
\label{prop:T-stable-quintics} For the pencil of curves $\langle s_{[\ell]}+4f_{T},s_{T}+f_{[\ell]}\rangle\subset|s_{[\ell]}+4f_{[\ell]}|$
in $\mathfrak{F}_{[\ell]}$ passing through $\mathfrak{p}$ generated
by $s_{[\ell]}+4f_{T}$ and $s_{T}+f_{[\ell]}$, the following hold: 

a) The exceptional divisor $\mathfrak{E}_{[\ell]}$ of the blow-up
$\alpha_{\mathfrak{p}}:\hat{\mathfrak{F}}_{[\ell]}\to\mathfrak{F}_{[\ell]}$
is a section of the proper transform of the pencil.

b) A rational normal quintic curve in $F_{p}$ whose stabilizer contains
$T$ is the image $\Gamma^{T}(v)$ by $\nu:\mathfrak{F}_{[\ell]}\to F_{p}$
of a member $\Upsilon^{T}(v)$, $v\in\mathfrak{E}_{[\ell]}\setminus\{q_{s_{[\ell]}},q_{f_{[\ell]}},\delta_{[\ell]}\}$,
of the pencil, other than $\Upsilon^{T}(q_{s_{[\ell]}}):=s_{[\ell]}+4f_{T}$,
$\Upsilon^{T}(q_{f_{[\ell]}}):=s_{T}+f_{[\ell]}$ and the unique member
$\Upsilon^{T}(\delta_{[\ell]})$ whose proper transform meets $\mathfrak{E}_{[\ell]}$
at $\delta_{[\ell]}$. In particular, every such curve intersects
$\ell$ at $p$ only.

c) All the curves $\Gamma^{T}(v)$, $v\in\mathfrak{E}_{[\ell]}\setminus\{q_{s_{[\ell]}},q_{f_{[\ell]}},\delta_{[\ell]}\}$,
other than $\Gamma_{p}=\nu_{*}\Upsilon_{p}$ have stabilizer $G_{\Gamma^{T}(v)}=T$. 
\end{prop}

\begin{proof}
The proper transform by $\nu$ of a $T$-stable rational normal quintic
$\Gamma\subset F_{p}$ other than $\Gamma_{p}$ is a smooth $T$-stable
curve $\Upsilon\in|s_{[\ell]}+4f_{[\ell]}|$ whose intersection with
$\Upsilon_{p}$ is a subscheme of length $5$ supported on the union
of the two $T$-fixed points $\mathfrak{p}=s_{[\ell]}\cap f_{[\ell]}$
and $r_{T}''=s_{T}\cap f_{T}$ of $\Upsilon_{p}$. If $\Upsilon$
does not pass through $\mathfrak{p}$ then it intersects the $T$-stable
curves $s_{[\ell]}$ and $f_{[\ell]}$ only at the $T$-fixed point
$r_{T}=s_{[\ell]}\cap f_{T}$ and $r_{T}'=s_{T}\cap f_{[\ell]}$ respectively.
In particular, $\Upsilon$ does not pass through $r_{T}''=\Upsilon_{p}\cap f_{T}$,
and hence $\Upsilon\cap\Upsilon_{p}=\emptyset$, which is absurd.
So $\Upsilon$ passes through $\mathfrak{p}$, whence intersects the
$T$-stable curve $s_{T}\sim s_{[\ell]}+3f_{[\ell]}$ at the $T$-fixed
point $r_{T}''$ only, necessarily with multiplicity $(\Upsilon\cdot s_{T})=4$.
Since $\Upsilon_{p}\cap s_{T}=4r_{T}''$, this implies that $(\Upsilon\cdot\Upsilon_{p})_{r_{T}''}\geq4$,
and hence, since $\Upsilon\cdot\Upsilon_{p}=5$ and $\Upsilon$ also
intersects $\Upsilon_{p}$ at $\mathfrak{p}$, that $(\Upsilon\cdot\Upsilon_{p})_{r_{T}''}=4$
and $(\Upsilon\cdot\Upsilon_{p})_{\mathfrak{p}}=1$. Since the restriction
homomorphism 
\[
\mathrm{res}_{\Upsilon_{p}}:H^{0}(\mathfrak{F}_{[\ell]},\mathcal{O}_{\mathfrak{F}_{[\ell]}}(s_{[\ell]}+4f_{[\ell]}))\to H^{0}(\Upsilon_{p},\mathcal{O}_{\mathfrak{F}_{[\ell]}}(s_{[\ell]}+4f_{[\ell]})|_{\Upsilon_{p}})\cong H^{0}(\mathbb{P}^{1},\mathcal{O}_{\mathbb{P}^{1}}(5))
\]
is surjective, with kernel generated by a section whose zero locus
equals $\Upsilon_{p}$, it follows that the $k$-vector subspace $V_{4r_{T}''+\mathfrak{p}}\subset H^{0}(\mathfrak{F}_{[\ell]},\mathcal{O}_{\mathfrak{F}_{[\ell]}}(s_{[\ell]}+4f_{[\ell]}))$
consisting of sections whose restrictions to $\Upsilon_{p}$ vanish
with order at least $4$ at $r_{T}''$ and order at least $1$ at
$\mathfrak{p}$ is $2$-dimensional. Thus, every smooth $T$-stable
curve $\Upsilon\in|s_{[\ell]}+4f_{[\ell]}|$ is a smooth member of
the pencil $|V_{4r_{T}''+\mathfrak{p}}|=\langle s_{[\ell]}+4f_{T},s_{T}+f_{[\ell]}\rangle$,
which also contains $\Upsilon_{p}$ among its members. 

Assertion a) follows from the observation that the proper transforms
of $s_{[\ell]}+4f_{T}$ and $s_{T}+f_{[\ell]}$ intersect $\mathfrak{E}_{[\ell]}$
transversally at the distinct points $q_{s_{[\ell]}}$ and $q_{f_{[\ell]}}$.
Since $\mathfrak{F}_{[\ell]}\cong\mathbb{F}_{3}$, a $T$-stable reducible
member $C$ of $|V_{4r_{T}''+\mathfrak{p}}|\subset|s_{[\ell]}+4f_{[\ell]}|$
other than $s_{[\ell]}+4f_{T}$ is the union of a $T$-stable section
$s\sim s_{[\ell]}+3f_{[\ell]}$ of $\pi$ and of $T$-stable fiber
$f$ of it such that $C\cap\Upsilon_{p}=4r_{T}''+\mathfrak{p}$. Since
$s_{[\ell]}\cdot s=0$ and $s\neq s_{[\ell]}$, $s$ is disjoint from
$s_{[\ell]}$. It follows that $f=f_{[\ell]}$ and that $s$ is the
unique section $s_{T}$ intersecting $\Upsilon_{p}$ with multiplicity
$4$ at $r_{T}''$. Thus, all members $\Upsilon^{T}(v)$ of $|V_{4r_{T}''+\mathfrak{p}}|$
other than $\Upsilon^{T}(q_{s_{[\ell]}})$ and $\Upsilon^{T}(q_{f_{[\ell]}})$
are irreducible, whence smooth, curves on $\mathfrak{F}_{[\ell]}$
and assertion b) then follows from Lemma \ref{lem:inverse-image-quintic-normalization}.
Finally, we observe that since $|V_{4r_{T}''+\mathfrak{p}}|$ contains
three $T$-stable elements $\Upsilon^{T}(q_{s_{[\ell]}})$, $\Upsilon^{T}(q_{f_{[\ell]}})$
and $\Upsilon_{p}$ every member of $|V_{4r_{T}''+\mathfrak{p}}|$
is $T$-stable and moreover that $\Upsilon_{p}$ is $G_{p}$-stable
by Lemma \ref{lem:Action-on-special-hyperplane-sections} b). 

Now assume that $\Upsilon^{T}(v)$ is an irreducible member of $|V_{4r_{T}''+\mathfrak{p}}|$
other than $\Upsilon_{p}$ whose stabilizer $G_{\Upsilon^{T}(v)}\subset G_{p}$
properly contains $T$. Then $G_{\Upsilon^{T}(v)}$ contains a unipotent
element, whence is equal to $G_{p}$. But then, the intersection $\Upsilon^{T}(v)\cap\Upsilon_{p}=4r_{T}''\cup\mathfrak{p}$
would be $G_{p}$-stable, which is impossible since, by Lemma \ref{lem:Action-on-special-hyperplane-sections}
a), $\mathfrak{p}$ is the unique $G_{p}$-fixed point of $\mathfrak{F}_{[\ell]}$. 
\end{proof}
Since the group $G_{p}$ acts transitively on the set of its maximal
tori, we obtain the following:
\begin{cor}
\label{cor:T-stable-quintics}Up to the induced action of $G_{p}$
on $F_{p}$, there exists a unique family $\Gamma^{m}(v)$, $v\in\mathbb{P}^{1}\setminus\{3\textrm{ points}\}$,
of rational normal quintic curves in $F_{p}$ whose stabilizers contain
a maximal torus of $G_{p}$. The family contains the unique $G_{p}$-stable
curve $\Gamma_{p}$ and every curve $\Gamma^{m}(v)$ other than $\Gamma_{p}$
has stabilizer $G_{\Gamma^{m}(v)}$ isomorphic to $\mathbb{G}_{m}$. 
\end{cor}

\indent$\bullet$ Case 2: $G_{\Gamma}$ does not contain a maximal
torus of $G_{p}$. 
\begin{prop}
\label{prop:Ga-stable-quintics} For the pencil of curves $\langle s_{[\ell]}+4f_{[\ell]},\Upsilon_{p}\rangle\subset|s_{[\ell]}+4f_{[\ell]}|$
in $\mathfrak{F}_{[\ell]}$ generated by $s_{[\ell]}+4f_{[\ell]}$
and $\Upsilon_{p}$, the following hold: 

a) Every rational normal quintic curve in $F_{p}$ whose stabilizer
is infinite and does not contain a maximal torus of $G_{p}$ is the
image $\Gamma^{a}(v)$ by $\nu\colon\mathfrak{F}_{[\ell]}\to F_{p}$
of a member $\Upsilon^{a}(v)$ of the pencil other than $s_{[\ell]}+4f_{[\ell]}$
and $\Upsilon_{p}$.

b) The stabilizer $G_{\Gamma^{a}(v)}$ is a subgroup of $G_{p}$ containing
the unipotent radical $G_{p}^{a}$ of $G_{p}$ such that $G_{\Gamma^{a}(v)}/G_{p}^{a}$
is isomorphic to the group $\mu_{4}$. 

c) The coset space of the action of $G_{p}$ on the set of pairs $(F_{p},\Gamma^{a}(v))$
is (non-canonically) isomorphic to $k^{*}/(k^{*})^{4}=1$.
\end{prop}

\begin{proof}
Let $\Gamma\subset F_{p}$ be a rational normal quintic curve whose
stabilizer $G_{\Gamma}\subset G_{p}$ is infinite but does not contain
a maximal torus of $G_{p}$. Then the neutral component $G_{\Gamma}^{0}$
of $G_{\Gamma}$ is equal to $G_{p}^{a}$ and the proper transform
of $\Gamma$ in $\mathfrak{F}_{[\ell]}$ is a smooth $G_{p}^{a}$-stable
curve $\Upsilon\sim s_{[\ell]}+4f_{[\ell]}$ which intersects $\Upsilon_{p}$,
$s_{[\ell]}$ and $f_{[\ell]}$ at $G_{p}^{a}$-fixed points only,
whence at $\mathfrak{p}$ only. Thus, $\Upsilon\cap\Upsilon_{p}=5\mathfrak{p}$
and the same argument as in the proof of Proposition \ref{prop:T-stable-quintics}
shows that $\Upsilon$ is an irreducible member of the pencil $\langle s_{[\ell]}+4f_{[\ell]},\Upsilon_{p}\rangle=|V_{5\mathfrak{p}}|$,
necessarily other than $\Upsilon_{p}$ by hypothesis. A reducible
member of $|V_{5\mathfrak{p}}|$ being the union of a $G_{p}^{a}$-stable
section $s$ of $\pi$ meeting $\Upsilon_{p}$ at $\mathfrak{p}$
only and of a multiple of the unique $G_{p}^{a}$-stable fiber $f_{[\ell]}$
of $\pi$, it follows that $s_{[\ell]}+4f_{[\ell]}$ is the unique
reducible member of $|V_{5\mathfrak{p}}|$. All members $\Upsilon^{a}(v)$
are thus smooth curves, and their proper transforms by $\alpha_{\mathfrak{p}}:\hat{\mathfrak{F}}_{[\ell]}\to\mathfrak{F}_{[\ell]}$
intersect the exceptional divisor $\mathfrak{E}_{[\ell]}$ at the
same point as that of $\Upsilon_{p}$. Assertion a) then follows from
Lemma \ref{lem:inverse-image-quintic-normalization}. Since $s_{[\ell]}+4f_{[\ell]}$
and $\Upsilon_{p}$ are $G_{p}^{a}$-stable, every curve in the pencil
$|V_{5\mathfrak{p}}|$ is $G_{p}^{a}$-stable. 

Let $T\subset G_{p}$ be a maximal torus and, as in the paragraph
preceding Proposition \ref{prop:T-stable-quintics}, let $s_{T}\sim s_{[\ell]}+3f_{[\ell]}$
be the unique $T$-stable section of $\pi:\mathfrak{F}_{[\ell]}\to T_{[\ell]}\mathfrak{C}$
intersecting $\Upsilon_{p}$ with multiplicity $4$ at its unique
$T$-fixed point $r_{T}''\in\Upsilon_{p}\setminus\{\mathfrak{p}\}$.
Since $s_{[\ell]}+4f_{[\ell]}$ and $\Upsilon_{p}$ are $T$-stable,
$f:\mathfrak{F}_{[\ell]}\dashrightarrow\mathbb{P}(V_{5\mathfrak{p}})$
is $T$-equivariant for a uniquely determined $T$-action on $\mathbb{P}(V_{5\mathfrak{p}})$,
not necessarily faithful, which fixes the points $q_{0}:=f_{*}(s_{[\ell]}+4f_{[\ell]})$
and $q_{\infty}:=f_{*}(\Upsilon_{p})$. On the other hand, since $(s_{[\ell]}+4f_{[\ell]})|_{s_{T}}=4r_{T}'$
and $\Upsilon_{p}|_{s_{T}}=4r_{T}''$, the restriction $\xi=f|_{s_{T}}:s_{T}\to\mathbb{P}(V_{5\mathfrak{p}})$
is a $T$-equivariant finite cover of degree $4$ totally ramified
over $\xi(r_{T}')=q_{0}$ and $\xi(r_{T}'')=q_{\infty}$ and \'etale
over $\mathbb{P}(V_{5\mathfrak{p}})\setminus\{q_{0},q_{\infty}\}$.
Identifying $T=\mathrm{Spec}(k[t^{\pm1}])$ with the open orbit $s_{T}\setminus\{r_{T}',r_{T}''\}=\mathrm{Spec}(k[u^{\pm1}])$
of its induced faithful action on $s_{T}$, we conclude that $\mathbb{P}(V_{5\mathfrak{p}})\setminus\{q_{0},q_{\infty}\}$
is $T$-equivariantly isomorphic to $\mathrm{Spec}(k[u^{\pm4}])$
on which $T$ acts by $t\cdot u^{4}=(tu)^{4}$. This shows that the
stabilizer $G_{\Upsilon^{a}(v)}$ of every curve $\Upsilon^{a}(v)$
in the pencil $|V_{5\mathfrak{p}}|$ other than $s_{[\ell]}+4f_{[\ell]}$
and $\Upsilon_{p}$ intersects $T$ along a subgroup isomorphic to
$\mu_{4}$ and identifies the coset space of the action of $T$ on
the set of pairs $(F_{p},\Gamma^{a}(v))$ with the orbit space $k^{*}/(k^{*})^{4}$
of the action of $T(k)$ on the $k$-rational points of $\mathbb{P}(V_{5\mathfrak{p}})\setminus\{q_{0},q_{\infty}\}$.
Since $T$ is a maximal torus of $G_{p}$, the induced homomorphism
$T\to G_{p}/G_{p}^{a}$ is an isomorphism, which gives assertion b). 

Assertion c) follows from the fact that since every curve $\Gamma^{a}(v)$
is $G_{p}^{a}$-stable, the coset space of pairs $(F_{p},\Gamma^{a}(v))$
under the action of $G_{p}\cong T\ltimes G_{p}^{a}$ identifies with
that under the action of $T$.
\end{proof}
Since the base field $k$ is by assumption algebraically closed, we
obtain the following:
\begin{cor}
\label{cor:Ga-stable-quintics}Up to the induced action of $G_{p}$
on $F_{p}$, there exists a unique rational normal quintic curve $\Gamma^{a}\subset F_{p}$
whose stabilizer contains the unipotent radical $G_{p}^{a}$ of $G_{p}$
as a subgroup of finite index. Moreover, $G_{\Gamma^{a}}$ is isomorphic
to the subgroup $\mathbb{G}_{a}\rtimes\mu_{4}$ of $G_{p}\cong\mathbb{G}_{a}\rtimes\mathbb{G}_{m}$. 
\end{cor}

\subsubsection{Parameter space and conclusions}

Since $\mathrm{Aut}(V_{5})$ acts transitively on the set of hyperplane
sections $F_{p}$ singular along a line $\ell_{p}$ of special type,
$p\in\mathcal{C}$, the stabilizers $G_{p}$ of such sections are
conjugate in $\mathrm{Aut}(V_{5}$). Thus, every rational normal quintic
curve $\Gamma$ with infinite stabilizer $G_{\Gamma}$ belongs to
the $\mathrm{Aut}(V_{5})$-orbit of precisely one of the curves in
Corollary \ref{cor:T-stable-quintics} and Corollary \ref{cor:Ga-stable-quintics}.
Moreover, given hyperplane sections $F_{p}$ and $F_{p'}$ singular
along lines of special types $\ell_{p}$ and $\ell_{p'}$ through
points $p,p'\in\mathcal{C}$, every automorphism of $V_{5}$ mapping
$p$ to $p'$ induces in a canonical way an isomorphism between the
blown-up surfaces $\hat{\mathfrak{F}}_{[\ell_{p}]}$ and $\hat{\mathfrak{F}}_{[\ell_{p'}]}$
which restricts to an isomorphism $\mathfrak{E}_{[\ell_{p}]}\to\mathfrak{E}_{[\ell_{p'}]}$
mapping the triple of points $(q_{s_{[\ell_{p}]}},q_{f_{[\ell_{p}]}},\delta_{[\ell_{p}]})$
onto the triple $(q_{s_{[\ell_{p'}]}},q_{f_{[\ell_{p'}]}},\delta_{[\ell_{p'}]})$
and the intersection point $\gamma_{p}$ of the proper transform of
$\Upsilon_{p}$ in $\hat{\mathfrak{F}}_{[\ell_{p}]}$ with $\mathfrak{E}_{[\ell_{p}]}$
onto the intersection point $\gamma_{p'}$ of the proper transform
of $\Upsilon_{p'}$ in $\hat{\mathfrak{F}}_{[\ell_{p'}]}$ with $\mathfrak{E}_{[\ell_{p'}]}$.
Thus, to fully determine the parameter space for rational normal quintic
curves $\Gamma\subset V_{5}$ with stabilizer equal to a maximal torus
of $\mathrm{Aut}(V_{5})$ up to the action of $\mathrm{Aut}(V_{5})$,
it suffices to choose a $G_{p}$-equivariant model for $F_{p}$ and
its normalization $\nu:\mathfrak{F}_{[\ell_{p}]}\to F_{p}$, to fix
an isomorphism $\xi:\mathbb{P}^{1}\to\mathfrak{E}_{[\ell_{p}]}$ mapping
the triple of points $(0,\infty,1)$ onto the triple of points $(q_{s_{[\ell_{p}]}},q_{f_{[\ell_{p}]}},\delta_{[\ell_{p}]})$
and to determine the point $\xi^{-1}(\gamma_{p})$. This is done in
the following example. 
\begin{example}
\label{exa:Explicit-model-Take-1} Let $F_{p}\subset\mathbb{P}_{[w_{0}:\cdots:w_{5}]}^{5}$
be the hyperplane section of $V_{5}$ singular along a line of special
type $\ell_{p}$ of $V_{5}$ and let $\Gamma_{p}\subset F_{p}$ be
the special rational normal quintic curve in $F_{p}$, see Lemma \ref{lem:non-normal-section-charac}.
To describe the normalization morphism $\nu:\mathfrak{F}_{\ell_{p}}\cong\mathbb{F}_{3}\to F_{p}$
in Proposition \ref{prop:Normalization-Special-Hyperplane-Section},
we endow $\mathbb{F}_{3}$ with bi-homogeneous coordinates $([x_{0}:x_{1}],[y_{0}:y_{1}])$,
where $\deg x_{0}=(1,0)$, $\deg x_{1}=(1,0)$, $\deg y_{0}=(-3,1)$
and $\deg y_{1}=(0,1)$, see e.g. \cite[Chapter 2]{Re97}, for which
$\rho_{3}:\mathbb{F}_{3}\to\mathbb{P}^{1}$ coincides with the projection
$([x_{0}:x_{1}],[y_{0}:y_{1}])\mapsto[x_{0}:x_{1}]$ and the exceptional
section with self-intersection $-3$ with the curve $s_{0}=\{y_{0}=0\}$.
We put $f_{0}=\{x_{0}=0\}=\rho_{3}^{-1}([0:1])$ and $\mathfrak{p}=s_{0}\cap f_{0}=([0:1],[0:1])$.
Since $\mathrm{Aut}(\mathbb{P}^{5})$ acts transitively on pairs consisting
of a rational normal quintic and a point of it, we can assume without
loss of generality that $\Gamma_{p}$ is the image $C_{5}$ of the
embedding $[x_{0}:x_{1}]\mapsto[x_{1}^{5}:x_{1}^{4}x_{0}:x_{1}^{3}x_{0}^{2}:x_{1}^{2}x_{0}^{3}:x_{1}x_{0}^{4}:x_{0}^{5}]$
of $\mathbb{P}^{1}$ and that $p=[1:0:0:0:0:0]\in\Gamma_{p}$ which
gives, since $\ell_{p}$ is the tangent line to $\Gamma_{p}$ at $p$,
that $\ell_{p}=\{w_{2}=\cdots=w_{5}=0\}\cong\mathbb{P}_{[w_{0}:w_{1}]}^{1}$.
We can assume further that the restriction isomorphisms $\nu|_{s_{0}}:s_{0}\cong\mathbb{P}_{[x_{0}:x_{1}]}^{1}\to\ell_{p}$
and $\nu|_{f_{0}}:f_{0}\cong\mathbb{P}_{[y_{0}:y_{1}]}^{1}\to\ell_{p}$
are given by $(\nu|_{s_{0}})^{*}w_{0}=x_{0}$, $(\nu|_{s_{0}})^{*}w_{1}=x_{1}$,
$(\nu|_{f_{0}})^{*}w_{0}=y_{0}$, $(\nu|_{f_{0}})^{*}w_{1}=y_{1}$.
Then, with the notation of Proposition \ref{prop:Normalization-Special-Hyperplane-Section}
and Remark \ref{rem:equalizer}, $\nu^{*}H^{0}(F_{p},\omega_{F_{p}}^{\vee})$
is the kernel of the homomorphism 
\[
(\nu|_{s_{0}}^{-1})^{*}\circ\mathrm{res}s_{0}-(\nu|_{f_{0}}^{-1})^{*}\circ\mathrm{res}f_{0}:H^{0}(\mathbb{F}_{3},\nu^{*}\omega_{F_{p}}^{\vee})\longrightarrow H^{0}(\ell_{p},\omega_{\ell_{p}}^{\vee}\otimes\mathcal{I}_{p}|_{\ell_{p}})^{\oplus2}
\]
Using the basis $\{x_{1}y_{1},x_{0}y_{1},x_{1}^{4}y_{0},x_{0}x_{1}^{3}y_{0},x_{0}^{2}x_{1}^{2}y_{0},x_{0}^{3}x_{1}y_{0},x_{0}^{4}y_{0}\}$
of $H^{0}(\mathbb{F}_{3},\nu^{*}\omega_{F_{p}}^{\vee})=H^{0}(\mathbb{F}_{3},\mathcal{O}_{\mathbb{F}_{3}}(s_{0}+4f_{0}))$
and $\{w_{0},w_{1}\}$ of $H^{0}(\ell_{p},\omega_{\ell_{p}}^{\vee}\otimes\mathcal{I}_{p}|_{\ell_{p}})\cong H^{0}(\ell_{p},\mathcal{O}_{\ell_{p}}(1))$,
it is straightforward to check that under our assumptions, $\nu^{*}H^{0}(F_{p},\omega_{F_{p}}^{\vee})$
equals the subspace $W'$ of $H^{0}(\mathbb{F}_{3},\nu^{*}\omega_{F_{p}}^{\vee})$
spanned by $\{x_{1}y_{1},x_{0}y_{1}+x_{1}^{4}y_{0},x_{0}x_{1}^{3}y_{0},x_{0}^{2}x_{1}^{2}y_{0},x_{0}^{3}x_{1}y_{0},x_{0}^{4}y_{0}\}$. 

Let $G=\mathrm{Spec}(k[a,\lambda^{\pm1}])=\mathbb{G}_{a}\rtimes\mathbb{G}_{m}$
with the maximal torus $T=\{a=0\}\cong\mathrm{Spec}(k[\lambda^{\pm1}])$
and group law given by $(a',\lambda')\cdot(a,\lambda)=(a+\lambda a',\lambda^{'}\lambda)$.
The unique faithful action of $G$ on $\mathbb{F}_{3}$ for which
$s_{0}$ and $f_{0}$ are $G$-stable with induced faithful actions
fixing $\mathfrak{p}=s_{0}\cap f_{0}$ and such that the sections
$s_{\infty}=\{y_{1}=0\}$ and the fiber $f_{\infty}=\{x_{1}=0\}$
are $T$-stable is given by the following formula
\begin{equation}
(a,\lambda)\cdot([x_{0}:x_{1}],[y_{0}:y_{1}])=([\lambda x_{0}:x_{1}+ax_{0}],[\lambda y_{0}:y_{1}+(ax_{1}^{3}+\frac{3}{2}a^{2}x_{0}x_{1}^{2}+a^{3}x_{0}^{2}x_{1}+\frac{1}{4}a^{4}x_{0}^{3})y_{0}]).\label{eq:Explicit-Action-Formula}
\end{equation}
It is straightforward to check that the unique irreducible $G$-stable
curve in the complete linear system $|s_{0}+4f_{0}|$ is 
\begin{equation}
\Upsilon_{p}=\{4x_{0}y_{1}-x_{1}^{4}y_{0}=0\}.\label{eq:Upsilon_p}
\end{equation}
Moreover, $W'$ is an irreducible $G$-module and the $G$-equivariant
morphism
\[
\psi:\mathbb{F}_{3}\to\mathbb{P}(W')=\mathbb{P}_{[w_{0}:\cdots:w_{5}]}^{5},([x_{0}:x_{1}],[y_{0}:y_{1}])\mapsto[x_{1}y_{1}:\frac{4}{5}(x_{0}y_{1}+x_{1}^{4}y_{0}):x_{0}x_{1}^{3}y_{0}:x_{0}^{2}x_{1}^{2}y_{0}:x_{0}^{3}x_{1}y_{0}:x_{0}^{4}y_{0}],
\]
maps $s_{0}$ and $f_{0}$ isomorphically onto $\ell_{p}$, restricts
to an isomorphism between $\mathbb{F}_{3}\setminus(s_{0}\cup f_{0})$
and its image and maps $\Upsilon_{p}$ isomorphically onto $C_{5}$.
The induced morphism $\nu:\mathbb{F}_{3}\to F_{p}$ is the normalization
of $F_{p}$ and we now consider it as our model of the normalization
$\nu:\mathfrak{F}_{[\ell_{p}]}\to F_{p}$ in Proposition \ref{prop:Normalization-Special-Hyperplane-Section}
for which $s_{[\ell_{p}]}=s_{0}$ and $f_{[\ell_{p}]}=f_{0}$. Note
that with the choices of bases above, $\psi:\mathbb{F}_{3}\to\mathbb{P}(W')$
is displayed as the composition of the $G$-equivariant closed embedding
$\mathbb{F}_{3}\to\mathbb{P}(H^{0}(\mathbb{F}_{3},\nu^{*}\omega_{F_{p}}^{\vee}))\cong\mathbb{P}^{6}$
followed by the $G$-equivariant linear projection $\mathbb{P}(H^{0}(\mathbb{F}_{3},\mathcal{O}_{\mathbb{F}_{3}}(s_{0}+4f_{0})))\dashrightarrow\mathbb{P}(W')$. 

The open neighborhood $U=\mathfrak{F}_{[\ell_{p}]}\setminus\{x_{1}y_{1}=0\}$
of the point $\mathfrak{p}=s_{0}\cap f_{0}=([0:1],[0:1])$ is isomorphic
to $\mathbb{A}^{2}$, with affine coordinates $x=x_{0}x_{1}^{-1}$
and $y=x_{1}^{3}y_{0}y_{1}^{-1}$ and we have $s_{0}\cap U=\{y=0\}$,
$f_{0}\cap U=\{x=0\}$ and $\Upsilon_{p}\cap U=\{4x-y=0\}$. Choosing
coordinates $[z_{0}:z_{1}]$ on $\mathbb{P}^{1}$ so that the restriction
of the blow-up $\alpha_{\mathfrak{p}}:\hat{\mathfrak{F}}_{[\ell_{p}]}\to\mathfrak{F}_{[\ell_{p}]}$
over $U$ is isomorphic to the subvariety $Z=\{xz_{1}+yz_{0}=0\}$
in $\mathbb{A}^{2}\times\mathbb{P}_{[z_{0}:z_{1}]}^{1}$ endowed with
the restriction $\alpha:Z\to\mathbb{A}^{2}$ of the first projection
provides an isomorphism $\xi:\mathbb{P}_{[z_{0}:z_{1}]}^{1}\to\mathfrak{E}_{[\ell_{p}]}$
which maps the triple of points $(0,\infty,1)=([0:1],[1:0],[1:1])$
onto the triple of points $(q_{s_{[\ell_{p}]}},q_{f_{[\ell_{p}]}},\delta_{[\ell_{p}]})$
and such that $\xi^{-1}(\gamma_{p})=[-4:1].$ 

With our choices of models and coordinates, the $T$-stable rational
normal quintics $\Gamma^{T}(v)\subset F_{p}$ of Proposition \ref{prop:T-stable-quintics}
are the images by $\nu:\mathfrak{F}_{[\ell_{p}]}\to F_{p}$ of the
curves 
\begin{equation}
\Upsilon^{T}(v)=\{vx_{0}y_{1}+x_{1}^{4}y_{0}=0\},\quad v\in\mathbb{P}^{1}\setminus\{0,1,\infty\}\label{eq:Gamma-Gm}
\end{equation}
with $\Upsilon^{MU}:=\Upsilon^{T}(-4)=\Upsilon_{p}$. We set in addition
\begin{equation}
\Gamma^{MU}:=\nu_{*}\Upsilon^{MU}=\Gamma_{p}.\label{eq:Gamma-MU}
\end{equation}

The $G_{p}^{a}$-stable rational normal quintics $\Gamma^{a}(v)\subset F_{p}$
of Proposition \ref{prop:Ga-stable-quintics} are the images of the
curves
\begin{equation}
\Upsilon^{a}(v)=\{4x_{0}y_{1}-x_{1}^{4}y_{0}+vx_{0}^{4}y_{0}=0\},\quad v\in k,\label{eq:Gamma-Ga}
\end{equation}
with $\Upsilon^{a}(0)=\Upsilon_{p}$, the curves $\Upsilon^{a}(v)$
with $v\in k^{*}$ having stabilizer $\mathbb{G}_{a}\rtimes\mu_{4}=\mathrm{Spec}(k[a,\lambda^{\pm1}]/(\lambda^{4}))$. 
\end{example}

In sum, we obtain the following classification: 
\begin{thm}
\label{thm:Quintics-Main-Theorem} The equivalence classes of pairs
$(V_{5},\Gamma)$ where $\Gamma\subset V_{5}$ is a rational normal
quintic curve with infinite stabilizer $G_{\Gamma}$ up to the action
of $\mathrm{Aut}(V_{5})$ are the following: 

1) The class of $(V_{5},\Gamma^{MU})$, where $G_{\Gamma^{MU}}$ is
a Borel subgroup of $\mathrm{Aut}(V_{5})$.

2) The class of $(V_{5},\Gamma^{a})$, where $\Gamma^{a}\neq\Gamma^{MU}$
and $G_{\Gamma^{a}}$ contains a non-trivial unipotent subgroup of
$\mathrm{Aut}(V_{5})$. Moreover, $G_{\Gamma^{a}}\cong\mathbb{G}_{a}\rtimes\mu_{4}$,
where the action of $\mu_{4}$ on $\mathbb{G}_{a}$ is given by the
natural injective homomorphism $\mu_{4}\to\mathrm{Aut}(\mathbb{G}_{a})\cong\mathbb{G}_{m}$. 

3) The class of $(V_{5},\Gamma^{m}(v))$, where $\Gamma^{m}(v)$ is
a member of a family of rational normal quintic curves $\Gamma^{m}(v)$,
$v\in\mathbb{P}^{1}\setminus\{0,1,\infty,-4\}$ whose stabilizers
$G_{\Gamma^{m}(v)}$ are all equal to a same maximal torus of $\mathrm{Aut}(V_{5})$. 
\end{thm}

\begin{rem}
Our parameter space $\mathbb{P}^{1}\setminus\{0,1,\infty,-4\}$ for
equivalence classes of rational normal quintic curves in $V_{5}$
whose stabilizer is a maximal torus differs from that $\mathbb{P}^{1}\setminus\{0,1,\frac{1}{5},\infty\}$
constructed by a different approach in \cite{KPS18}. Nevertheless,
we leave to the reader to check that the automorphism $\mathbb{P}^{1}\ni v\mapsto u^{4}=\frac{v}{4+v}\in\mathbb{P}^{1}$,
induces a one-to-one correspondence between the rational quintic curves
$\Gamma^{m}(v)$, $v\in\mathbb{P}^{1}\setminus\{0,1,\infty,-4\}$
and the rational normal quintic curves $Z_{m}(u^{4})$, $u^{4}\in\mathbb{P}_{u^{4}}^{1}\setminus\{0,1,\frac{1}{5},\infty\}$,
of \cite[Example 5.3.4 and Remark 5.3.8]{KPS18}. 
\end{rem}

\section{\protect\label{sec:X_22-stuff}Application to smooth prime Fano threefolds
of degree $22$ with infinite automorphism groups}

Hereafter, we consider every smooth prime threefold $X$ of degree
$22$ as a subvariety of $\mathbb{P}^{13}$ via its anti-canonical
embedding $\varphi_{|\mathcal{\omega}_{X}^{\vee}|}:X\hookrightarrow\mathbb{P}(H^{0}(X,\omega_{X}^{\vee}))\cong\mathbb{P}^{13}$.
A line in $X$ is a line $Z\cong\mathbb{P}^{1}$ in $\mathbb{P}^{13}$
which is contained in $X$. By \cite[Lemma 1]{Isk89}, the conormal
sheaf $\mathcal{C}_{Z/X}$ of a line $Z$ in $X$ is isomorphic either
to $\mathcal{O}_{Z}\oplus\mathcal{O}_{Z}(1)$, in which case $Z$
is said to be of general type, or to $\mathcal{O}_{Z}(-1)\oplus\mathcal{O}_{Z}(2)$,
in which case $Z$ is said to be of special type. The Hilbert scheme
$\mathcal{H}_{X}$ of lines in $X$ is of pure dimension one with
singular locus supported by the subset of lines of special type \cite[Proposition 1]{Isk89}. 

\subsection{Double projection from a line}

Given a line $Z$ in a smooth prime Fano threefold $X$ of degree
$22$ the linear system $|\omega_{X}^{\vee}\otimes{\mathscr{I}}_{Z}^{2}|$
of hyperplane sections of $X$ singular along $Z$ defines a rational
map $\psi_{Z}:X\dashrightarrow\mathbb{P}^{6}$ with image equal to
the quintic del Pezzo threefold $V_{5}$, called the \emph{double
projection from} $Z$, which can be described as follows: 
\begin{thm}
\label{thm:sarkisov} There exists a Sarkisov link 
\[
\xymatrix{ &  & F'\subseteq X'\ar@{.>}[r]^{\chi}\ar@{->}[lld]_{\sigma} & X^{+}\supseteq S^{+}\ar@{->}[rrd]^{\tau}\\
Z\subset S\subset X\ar@{-->}[rrrrr]^{\psi_{Z}} &  &  &  &  & V_{5}\supset F\supset\Gamma
\\
}
\]
where $\sigma:X'\to X$ is the blow-up with center $Z$ and exceptional
divisor $F'$, $\chi:X'\dasharrow X^{+}$ is a flop (that is not identity),
$\tau:X^{+}\to V_{5}$ is the contraction of the proper transform
$S^{+}$ of the unique hyperplane section $S\in|\omega_{X}^{\vee}\otimes{\mathscr{I}}_{Z}^{3}|$
to a rational normal quintic curve $\Gamma$ on the quintic del Pezzo
threefold $V_{5}$. The birational inverse $\varphi_{\Gamma}$ of
$\psi_{Z}$ is given by the linear system of cubic hypersurfaces of
$V_{5}$ singular along $\Gamma$. Moreover, the following hold:

$\quad$ a) The flopping locus of $\chi$ is the union of the proper
transforms by $\sigma$ of lines on $X$ intersecting $Z$, and of
the exceptional section of  the divisor $F'$ if the line $Z$ is
of special type and its flopped locus is the union of the proper transforms
by $\tau$ of the lines on $V_{5}$ which are bisecant to $\Gamma$. 

$\quad$ b) The proper transform $F$ of $F'$ on $V_{5}$, which
is the unique hyperplane section of $V_{5}$ containing $\Gamma$,
is a normal surface if and only if $Z$ is of general type.
\end{thm}

\begin{proof}
All assertions but b) can be found in \cite{Isk89,Pr92} and in \cite[Lemma 5.25]{KPS18}
for the description of the flopping and flopped locus of $\chi$.
Let $\bar{X}$ be the midpoint of the flop $\chi:X'\dasharrow X^{+}$,
let $\phi:X'\to\bar{X}$ and $\phi^{+}:X^{+}\to\bar{X}$ be the respective
small anti-canonical contractions so that $\chi=(\phi^{+})^{-1}\circ\phi$,
and let $\bar{F}=\phi_{*}F'=\phi_{*}^{+}F^{+}$, where $F^{+}=\tau_{\ast}^{-1}F$.
The morphism $\sigma|_{F'}:F'=\mathbb{P}(\mathcal{C}_{Z/X})\to Z$
identifies $F'$ with the Hirzebruch surface $\rho_{e}:\mathbb{F}_{e}\to\mathbb{P}^{1}$
with $e=1$ if $Z$ is of general type and $e=3$ otherwise. By \cite[Proposition 3 (ii)-(iii)]{Isk89},
the morphism $\phi$ is given by the complete linear system $|-K_{X'}|$
and 
\begin{equation}
\bar{F}\cong\begin{cases}
\mathbb{F}_{1} & \textrm{if }Z\textrm{ is of general type}\\
\mathbb{P}(1,1,3) & \textrm{if }Z\textrm{ is of special type.}
\end{cases}\label{eq:F-bar}
\end{equation}
If $Z$ is of general type, it follows from assertion a) that the
proper transform in $V_{5}$ of a general fiber $f'$ of $\sigma|_{F'}$
is a nonsingular conic $C$ contained in $F$, which implies that
$F$ is normal. Indeed, otherwise, by \cite{FT92}, the normalization
$\nu:\tilde{F}\to F$ of $F$ would be isomorphic to a Hirzebruch
surface $\rho_{a}:\mathbb{F}_{a}\to\mathbb{P}^{1}$ for some $a\in\{1,3\}$
and the proper transform $\nu^{*}(H)$ of a general hyperplane section
$H$ of $F$ would be linearly equivalent to $s_{a}+\tfrac{1}{2}(5+a)f_{a}$.
But then $\nu^{*}(C)$ would be an irreducible and reduced movable
(hence nef) curve in $\mathbb{F}_{a}$ such that $\nu^{*}(C)\cdot\nu^{*}(H)=2$,
which is impossible. 

Conversely, assume that $F$ is normal. Then $F^{+}$ is normal as
well. Indeed, $F^{+}$ being a Cartier divisor on the smooth threefold
$X^{+}$, it is Cohen-Macaulay. On the other hand, the normality of
$F$ and $\bar{F}$ implies that any curve $C^{+}\subseteq{\rm Sing}(F^{+})$
has to be contracted by $\tau|_{F^{+}}$ and $\phi^{+}|_{F^{+}}$,
which is impossible. Thus, $F^{+}$ is regular in codimension $1$,
hence normal. It follows from adjunction formula that $-K_{F^{+}}\sim(-K_{X^{+}}-F^{+})|_{F^{+}}\sim(\tau|_{F^{+}})^{\ast}(-K_{F})$.
The morphism $\tau|_{F^{+}}:F^{+}\to F$ is thus crepant and since
$F$ has canonical singularities by \cite{HW}, it follows in turn
that $F^{+}$ has canonical singularities. On the other hand, since
$\tau^{*}(-K_{V_{5}})$ is $\phi^{+}$-ample and $-K_{F^{+}}\sim-\frac{1}{2}\tau^{*}K_{V_{5}}|_{F^{+}}$,
it follows that $-K_{F^{+}}$ is $\phi^{+}|_{F^{+}}$-ample. This
implies that $\bar{F}$ has canonical singularities as well, whence,
by (\ref{eq:F-bar}), that $Z$ is of general type. 
\end{proof}
We record the following functoriality of the double projection from
a line (see also \cite[Remark 5.2.6]{KPS18}): 
\begin{cor}
\label{cor:Functoriality-double-projection}Let $X_{i}$, $i=1,2$,
be smooth prime Fano threefold of degree $22$ and let
\[
\psi_{Z_{i}}:X_{i}\dashrightarrow V_{5,i}\quad\textrm{and}\quad\varphi_{\Gamma_{i}}:V_{5,i}\dashrightarrow X_{i}\quad i=1,2
\]
be the mutually inverse birational maps associated to a line $Z_{i}$
on $X_{i}$ as in Theorem \ref{thm:sarkisov}. Then the map 
\[
\epsilon:\mathrm{Isom}((X_{1},Z_{1}),(X_{2},Z_{2}))\to\mathrm{Isom}((V_{5,1},\Gamma_{1}),(V_{5,2},\Gamma_{2})),\;h\mapsto h':=\psi_{Z_{2}}\circ h\circ\varphi_{\Gamma_{1}}
\]
 is a well-defined bijection with inverse $h'\mapsto\varphi_{\Gamma_{2}}\circ h'\circ\psi_{Z_{1}}$. 
\end{cor}

\begin{proof}
An isomophism of pairs $h:(X_{1},Z_{1})\to(X_{2},Z_{2})$ induces
isomorphisms of $k$-vector spaces 
\[
h_{2}^{*}:H^{0}(X_{2},\omega_{X_{2}}^{\vee}\otimes{\mathscr{I}}_{Z_{2}}^{2})\to H^{0}(X_{1},\omega_{X_{1}}^{\vee}\otimes{\mathscr{I}}_{Z_{1}}^{2})
\]
By definition of the double projection, the rational map $\psi_{Z_{2}}\circ h$
is equal to $\mathbb{P}(h_{2}^{*})\circ\psi_{Z_{1}}$, where 
\[
\mathbb{P}(h_{2}^{*}):\mathbb{P}H^{0}(X_{1},\omega_{X_{1}}^{\vee}\otimes{\mathscr{I}}_{Z_{1}}^{2})\to\mathbb{P}H^{0}(X_{2},\omega_{X_{2}}^{\vee}\otimes{\mathscr{I}}_{Z_{2}}^{2})
\]
is an isomorphism which maps the image $V_{5,1}$ of $\psi_{Z_{1}}$
onto the image $V_{5,2}$ of $\psi_{Z_{2}}$. On the other hand, since
$S_{i}$ is the unique hyperplane section of $X_{i}$ having multiplicity
at least $3$ along $Z_{i}$, we have $h(S_{1})=S_{2}$ from which
it follows, by definition of $\Gamma_{i}\subset V_{5,i}$ as the image
of $S_{i}$ by $\psi_{Z_{i}}$, that $\mathbb{P}(h_{2}^{*})$ maps
$\Gamma_{1}$ onto $\Gamma_{2}$. The restriction $h':V_{5,1}\to V_{5,2}$
of $\mathbb{P}(h_{2}^{*})$ is thus an isomorphism of pairs $(V_{5,1},\Gamma_{1})\to(V_{5,2},\Gamma_{2})$
such that $h'\circ\psi_{Z_{1}}=\psi_{Z_{2}}\circ h$ as birational
maps, which shows that $\epsilon$ is well-defined and injective.
The same argument applied to the linear systems defining the birational
maps $\varphi_{\Gamma_{i}}:V_{5,i}\dashrightarrow X_{i}$, $i=1,2$,
provides a well-defined injective map 
\[
\epsilon':\mathrm{Isom}((V_{5,1},\Gamma_{1}),(V_{5,2},\Gamma_{2}))\to\mathrm{Isom}((X_{1},Z_{1}),(X_{2},Z_{2})),\;h'\mapsto\varphi_{\Gamma_{2}}\circ h'\circ\psi_{Z_{1}}
\]
 such that $\epsilon'\circ\epsilon=\mathrm{id}_{\mathrm{Isom}((X_{1},Z_{1}),(X_{2},Z_{2}))}$
and $\epsilon\circ\epsilon'=\mathrm{id}_{\mathrm{Isom}((V_{5,1},\Gamma_{1}),(V_{5,2},\Gamma_{2}))}$. 
\end{proof}
By combining Theorem \ref{thm:sarkisov} with the descriptions of
lines in $V_{5}$ and of hyperplane sections of $V_{5}$ singular
along a line of special type given in subsection \ref{subsec:Lines-V5}
and subsection \ref{subsec:Normalization-Special-Hyperplane-Section}
respectively, we obtain through suitable double projections the following
correspondence between lines in $X$ and in $V_{5}$: 
\begin{prop}
\label{prop:Hilbert-scheme-correspondence}Let $X$ be a smooth prime
Fano threefold of degree $22$ and let $\psi_{Z}:X\dashrightarrow V_{5}$
be the double projection from a line $Z$ in $X$. Assume that the
hyperplane section containing the base locus $\Gamma$ of $\psi_{Z}^{-1}$
is the hyperplane section $F_{p}$ singular along a line of special
type $\ell=\ell_{p}$ of $V_{5}$, $p\in\mathcal{C}$, and that $\Gamma$
intersects $\ell$ at $p$. Then $Z$ is a line of special type and
following hold:

1) The line $Z$ is the unique line of $X$ contained in $S$.

2) Every line $Z'\neq Z$ in $X$ is disjoint from $Z$, in particular
$Z$ is the base locus of $\psi_{Z}$.

3) The map which associates to a line $Z'\subset X$ other than $Z$
its proper transform $\ell_{Z'}\subset V_{5}$ by $\psi_{Z}$ induces
a one-to-one correspondence between lines on $X$ other than $Z$
and lines on $V_{5}$ not contained in $F_{p}$ which intersect $\Gamma$. 

4) Moreover, a line $Z'\subset X$ other than $Z$ is of special type
if and only if $\ell_{Z'}$ is a line of special type which meets
$\Gamma$ at a tangency point $q\in\Gamma\setminus\{p\}$ of the curves
$\Gamma$ and $\Gamma_{p}$ in $F_{p}$. In particular, if $\Gamma=\Gamma_{p}\subset F_{p}\cap\mathcal{S}$
then the lines of special type other than $Z$ on $X$ are the proper
transforms of the lines of special type in $V_{5}$ passing through
the points of $\Gamma_{p}\setminus\{p\}$. 
\end{prop}

\begin{proof}
Since, by assumption, $F_{p}$ is non-normal, $Z$ is a line of special
type by Theorem \ref{thm:sarkisov} b). Assertions 1), 2) and 3) can
be derived from the same arguments as in \cite[Lemma 5.4.2 and Proposition 5.4.3]{KPS18}.
Nevertheless, since the aforementioned results are stated under assumptions
slightly different than ours, we provide a brief self-contained proof.
Since $Z$ is a line of special type, by Theorem \ref{thm:sarkisov}
a) the flopping locus of $\chi:X'\dashrightarrow X^{+}$ is the union
of the exceptional section of $\sigma|_{F'}:F'=\mathbb{P}(\mathcal{C}_{Z/X})\to Z$
and of the proper transforms $\sigma_{*}^{-1}Z'$ in $X'$ of the
lines $Z'\neq Z$ in $X$ which meet $Z$, and its flopped locus consists
of the proper transforms in $X^{+}$ of lines $\ell'$ in $V_{5}$
which are bisecant to $\Gamma\subset F_{p}$. Since $F_{p}$ is a
hyperplane section of $V_{5}$, a line $\ell'\subset V_{5}$ bisecant
to $\Gamma$ is necessarily contained in $F_{p}$. So, by Proposition
\ref{prop:Normalization-Special-Hyperplane-Section}, $\ell'$ is
the image by the normalization morphism $\nu:\mathfrak{F}_{[\ell]}\to F_{p}$
of a fiber $f'$ of $\pi:\mathfrak{F}_{[\ell]}\to T_{[\ell]}\mathfrak{C}$.
On the other hand, by Lemma \ref{lem:inverse-image-quintic-normalization},
a rational normal quintic curve $\Gamma\subset F_{p}$ passing through
$p$ is the image by $\nu$ of an irreducible and reduced member $\Upsilon$
of the linear system $|s_{[\ell]}+4f_{[\ell]}|$ passing through $\mathfrak{p}=s_{[\ell]}\cap f_{[\ell]}$.
Proposition \ref{prop:Normalization-Special-Hyperplane-Section} d)
then implies that the image $\nu(f')$ of every fiber $f'$ of $\pi:\mathfrak{F}_{[\ell]}\to T_{[\ell]}\mathfrak{C}$
other than $f_{[\ell]}$ intersects $\Gamma$ transversally at the
unique point $\nu(f'\cap\Upsilon)\in F_{p}\setminus\ell$. So $\ell=\nu(f_{[\ell]})$
is the unique bisecant line to $\Gamma$ and hence, the flopping locus
of $\chi$ consists only of the exceptional section of $\sigma|_{F'}:F'\to Z$.
This implies that every line $Z'$ in $X$ other than $Z$ is disjoint
from $Z$, which proves assertion 2), and that its proper transform
$\sigma_{*}^{-1}Z'$ on $X'$ is disjoint from $F'$. 

Since the rational map $\tau\circ\chi:X'\dashrightarrow V_{5}$ is
given by the complete linear system $|-K_{X'}-F'|$, we have 
\begin{equation}
-\frac{1}{2}\tau^{*}K_{V_{5}}\cdot\chi_{*}(\sigma_{*}^{-1}Z')=(-K_{X'}-F')\cdot\sigma_{*}^{-1}Z'=1.\label{eq:transform-line}
\end{equation}
This implies that $\chi_{\ast}(\sigma_{\ast}^{-1}Z')$ is not contained
in the exceptional divisor $S^{+}$ of $\tau$ over the rational normal
quintic curve $\Gamma$, which gives assertion 1). Since $K_{X^{+}}=\tau^{*}K_{V_{5}}+S^{+}$
and $K_{X^{+}}\cdot\chi_{*}(\sigma_{*}^{-1}Z')=K_{X'}\cdot\sigma_{*}^{-1}Z'$,
(\ref{eq:transform-line}) implies in turn that $\chi_{*}(\sigma_{*}^{-1}Z')$
intersects $S^{+}$ transversally in a unique point, whence that the
proper transform of $Z'$ in $V_{5}$ is a line $\ell_{Z'}$ in $V_{5}$
which intersects $F$ transversally at a point $q$ of the image $\Gamma$
of $S^{+}$ by $\tau:X^{+}\to V_{5}$. Note that since $\ell$ is
the unique line of $V_{5}$ passing through $p\in\mathcal{C}$, we
have $q\in\Gamma\setminus\{p\}$. Conversely, the proper transform
by $\psi_{Z}^{-1}$ of a line $\ell'\subset V_{5}$ intersecting $\Gamma$
and not contained in $F$ is a line $Z_{\ell'}\subset X$ not contained
in $S$ and disjoint from $Z$. This gives assertion 3). 

To prove assertion 4), consider a line $Z'\subset X$ other than $Z$
with conormal sheaf $\mathcal{C}_{Z'/X}$ and its proper transform
$(\psi_{Z})_{*}Z'$, which is a line $\ell'$ in $V_{5}$ not contained
in $F_{p}$ and intersecting $\Gamma$ at a point $q\in\Gamma\setminus\{p\}$.
Let $\beta_{Z'}:\hat{X}\to X$ be the blow-up of $X$ along $Z'$
with exceptional divisor $\rho_{Z'}:E_{Z'}\cong\mathbb{P}(\mathcal{C}_{Z'/X})\to Z'$
and let $\beta_{\ell'}:\hat{V}_{5}\to V_{5}$ be the blow-up of $V_{5}$
along $\ell'$ with exceptional divisor $\rho_{\ell'}:E_{\ell'}\cong\mathbb{P}(\mathcal{C}_{\ell'/V_{5}})\to\ell'$,
where $\mathcal{C}_{\ell'/V_{5}}$ is the conormal sheaf of $\ell'$
in $V_{5}$. The birational map $\mathrm{e}:E_{\ell'}\dashrightarrow E_{Z'}$
defined by the restriction of the inverse of the birational map $\hat{\psi}_{Z}:\hat{X}\dashrightarrow\hat{V}_{5}$
induced by $\psi_{Z}$ is an elementary transformation between Hirzebruch
surfaces which consists of the blow-up of the intersection point of
$E_{\ell'}$ with the proper transform $\hat{\Gamma}$ of $\Gamma$
in $\hat{V}_{5}$ followed by the contraction of the proper transform
of $\rho_{\ell'}^{-1}(q).$ Since
\[
E_{\ell'}\cong\begin{cases}
\mathbb{F}_{0} & \textrm{if }\ell'\textrm{ is of general type}\\
\mathbb{F}_{2} & \textrm{if }\ell'\textrm{ is of special type}
\end{cases}\quad\textrm{and}\quad E_{Z'}\cong\begin{cases}
\mathbb{F}_{1} & \textrm{if }Z'\textrm{ is of general type}\\
\mathbb{F}_{3} & \textrm{if }Z'\textrm{ is of special type, }
\end{cases}
\]
it follows that $Z'$ is of special type if and only if $\ell'$ is
of special type and $\hat{\Gamma}$ intersects $E_{\ell'}$ at the
intersection point $q'$ of $\rho_{\ell'}^{-1}(q)$ with the exceptional
section $s_{2}$ of $\rho_{\ell'}:E_{\ell'}\cong\mathbb{F}_{2}\to\ell'$
with self-intersection number $-2$. By subsection \ref{subsec:Lines-V5},
a line $\ell'\subset V_{5}$ of special type other than $\ell$ intersects
$F_{p}$ transversally at a point of the special rational normal quintic
curve $\Gamma_{p}\subset\mathcal{S}\cap F_{p}=\Gamma_{p}\cup5\ell$
other than $p$ (see Lemma \ref{lem:non-normal-section-charac}).
So $q=\ell'\cap F_{p}\in\Gamma\setminus\{p\}$ belongs to the intersection
of $\Gamma$ and $\Gamma_{p}$. Since, by Lemma \ref{lem:Propoer-transforms-special-quintics},
the proper transform of $\Gamma_{p}$ in $\hat{V_{5}}$ intersects
$E_{\ell'}$ at the point $q'=s_{2}\cap\rho_{\ell'}^{-1}(q)$, the
property that $\hat{\Gamma}$ intersects $E_{\ell'}$ at $q'$ is
equivalent to the property that $q$ is a tangency point of $\Gamma$
and $\Gamma_{p}$ in $F_{p}$. 
\end{proof}

\subsection{Isomorphism types of prime Fano threefold of degree $22$ with infinite
automorphism groups}

Recall from \cite[$\S$ 4.c.]{Gro61} that the automorphism group ${\rm Aut}(X)$
of a smooth projective variety $X$ is an algebraic group scheme locally
of finite type, in particular its neutral component $\mathrm{Aut}^{0}(X)$
is an algebraic group. For a smooth prime Fano threefold $X$ of degree
$22$, the anti-canonical embedding $\varphi_{|\mathcal{\omega}_{X}^{\vee}|}:X\hookrightarrow\mathbb{P}(H^{0}(X,\omega_{X}^{\vee}))\cong\mathbb{P}^{13}$
is $\mathrm{Aut}(X)$-equivariant with respect to the action of $\mathrm{Aut}(X)$
on $\mathbb{P}^{13}$ induced by the canonical $\mathrm{Aut}(X)$-linearization
of $\omega_{X}^{\vee}$, identifying $\mathrm{Aut}(X)$ with the affine
algebraic subgroup $\mathrm{Aut}(\mathbb{P}^{13},\varphi_{|\mathcal{\omega}_{X}^{\vee}|}(X))$
of $\mathrm{Aut}(\mathbb{P}^{13})=\mathrm{PGL}_{14}$. 
\begin{lem}
\label{prop:line} Let $X$ be a smooth prime Fano threefold of degree
$22$ with infinite automorphism group ${\rm Aut}(X)$ and let $B$
be a nontrivial connected solvable subgroup of ${\rm Aut}(X)$. Then
the following hold:

1) $B$ is isomorphic either to $\mathbb{G}_{a}$, or to $\mathbb{G}_{m}$
or to $\mathbb{G}_{a}\rtimes\mathbb{G}_{m}$.

2) $X$ contains a $B$-stable line of special type and moreover,
all $B$-stable lines in $X$ are of special type.
\end{lem}

\begin{proof}
The result is implicitly stated in \cite{KPS18}, we provide a different
proof. The Hilbert scheme of lines on $X$ being a projective curve,
the existence of a $B$-stable line $Z$ follows from Borel's fixed
point theorem. The first assertion follows from Theorem \ref{thm:Quintics-Main-Theorem}
and Corollary \ref{cor:Functoriality-double-projection}, which identifies,
through the double projection $\psi_{Z}:X\dashrightarrow V_{5}$,
the group $B\subset\mathrm{Aut}(X,Z)$ with a nontrivial connected
solvable subgroup of the stabilizer $\mathrm{Aut}(V_{5},\Gamma)$
of a rational normal quintic curve $\Gamma\subset V_{5}$. Moreover,
by Lemma \ref{lem:non-normal-section-charac}, the unique hyperplane
section $F$ of $V_{5}$ containing $\Gamma$ is the non-normal section
swept out by lines meeting a line of special type in $V_{5}$ and
then, Theorem \ref{thm:sarkisov} b) implies that $Z$ is of special
type. 
\end{proof}
Putting all the pieces together, we obtain the following classification: 
\begin{thm}
\label{prop:Second-Isomorphism-Classification} Up to isomorphism,
the following hold:

1) There exists a unique threefold $X_{22}^{MU}$ such that $\mathrm{Aut}^{0}(X_{22}^{MU})$
contains a Borel subgroup isomorphic to $\mathbb{G}_{a}\rtimes\mathbb{G}_{m}$.
Moreover, every line on $X_{22}^{MU}$ is of special type and $\mathrm{Aut}(X_{22}^{MU})\cong\mathrm{PGL}_{2}$.

2) There exists a unique threefold $X_{22}^{a}$ such that $\mathrm{Aut}^{0}(X_{22}^{a})$
contains a Borel subgroup isomorphic to $\mathbb{G}_{a}$. Moreover,
$X_{22}^{a}$ contains a unique line of special type and $\mathrm{Aut}(X_{22}^{a})\cong\mathbb{G}_{a}\rtimes\mu_{4}$,
where the action of $\mu_{4}$ on $\mathbb{G}_{a}$ is given by the
natural injective homomorphism $\mu_{4}\to\mathrm{Aut}(\mathbb{G}_{a})\cong\mathbb{G}_{m}$. 

3) There exist pairwise non-isomorphic threefolds $X_{22}^{m}(v)$,
$v\in\mathbb{P}^{1}\setminus\{0,1,\infty,-4\}$, such that $\mathrm{Aut}^{0}(X_{22}^{m}(v))$
contains a Borel subgroup isomorphic to $\mathbb{G}_{m}$. Each threefold
$X_{22}^{m}(v)$ contains exactly two lines of special type and $\mathrm{Aut}(X_{22}^{m}(v))\cong\mathbb{G}_{m}\rtimes\mu_{2}$,
where the group $\mu_{2}$ is generated by an involution $\theta$
which exchanges the two lines of special type in $X_{22}^{m}(v)$.
\end{thm}

\begin{proof}
Let $B$ be a nontrivial Borel subgroup of $\mathrm{Aut}(X)$ and
let $Z$ be a $B$-stable line of special type given by Lemma \ref{prop:line}
2). By Theorem \ref{thm:sarkisov} a) and Corollary \ref{cor:Functoriality-double-projection},
the double projection $\psi_{Z}:X\dashrightarrow V_{5}$ from $Z$
identifies the stabilizer $G_{Z}=\mathrm{Aut}(X,Z)$ of $Z$ with
the stabilizer $G_{\Gamma}=\mathrm{Aut}(V_{5},\Gamma)$ of a rational
normal quintic curve $\Gamma\subset V_{5}$. By the classification
established in subsection \ref{subsec:Classification-Quintic-Curves-Infinite-Stabilizers},
the unique hyperplane section $F$ of $V_{5}$ containing $\Gamma$
is a hyperplane section $F_{p}$, $p\in\mathcal{C}$, singular along
a line $\ell_{p}$ of special type of $V_{5}$, $\Gamma$ passes through
$p$ and the neutral component $G_{\Gamma}^{0}$ of $G_{\Gamma}$
is either equal to stabilizer $G_{p}\cong\mathbb{G}_{a}\rtimes\mathbb{G}_{m}$
of $p$, or to a maximal torus $T$ of $G_{p}$ or to the unipotent
radical of $G_{p}$. Since all these groups are connected and solvable,
$G_{Z}^{0}$ is a connected and solvable subgroup of $\mathrm{Aut}(X)$
containing $B$, whence, by the maximality of $B$ as a connected
solvable subgroup of $\mathrm{Aut}(X)$, is equal to $B$. Theorem
\ref{thm:Quintics-Main-Theorem} then gives $X_{22}^{MU}$, $X_{22}^{a}$
and $X_{22}^{m}(v)$, $v\in\mathbb{P}^{1}\setminus\{0,1,\infty,-4\}$,
for the isomorphism types for $X$ and the corresponding subgroups
$B$. We now prove the remaining assertions in each case:

\medskip

$\bullet$ Case $\Gamma=\Gamma^{a}$. By Proposition \ref{prop:Ga-stable-quintics}
and Corollary \ref{cor:Ga-stable-quintics}, the pair $(V_{5},\Gamma^{a})$
is unique up to isomorphism and $\Gamma^{a}\cap\Gamma_{p}=\{p\}$.
Moreover, since $\Gamma$ and $\Gamma_{p}$ meet at $p$ only, Proposition
\ref{prop:Hilbert-scheme-correspondence} 4) implies that every line
in $X_{22}^{a}$ other than $Z$ is of general type. Corollary \ref{cor:Functoriality-double-projection}
implies that the isomorphism class of the pair $(X_{22}^{a},Z)$ equals
that of the threefold $X_{22}^{a}$ and that $\mathrm{\mathrm{Aut}}(X_{22}^{a})=\mathrm{Aut}(X_{22}^{a},Z)$
identifies with the stabilizer $G_{\Gamma^{a}}\cong\mathbb{G}_{a}\rtimes\mu_{4}$
of $\Gamma^{a}\subset V_{5}$, see Corollary \ref{cor:Ga-stable-quintics}. 

\medskip

$\bullet$ Case $\Gamma=\Gamma^{T}(v)$, $v\in\mathbb{P}^{1}\setminus\{0,1,\infty\}$.
If $\Gamma=\Gamma_{p}\subset\mathcal{S}\cap F_{p}$ then $B=G_{p}\cong\mathbb{G}_{a}\rtimes\mathbb{G}_{m}$
and otherwise $B$ is some maximal torus $T$ of $G_{p}$. In the
first case, we let $T$ be any maximal torus of $G_{p}$. 

We first complete the description of the lines in $X$. Let $q'\in\Gamma_{p}$
be the unique $T$-stable point of $\Gamma_{p}$ other than $p$,
that is, with the notation of Case (1) in subsection \ref{subsec:Classification},
$q'$ is the image of the point $r_{T}''$ by the normalization morphism
$\nu:\mathfrak{F}_{[\ell]}\to F_{p}$. The unique line $\ell_{p'}$,
$p'\in\mathcal{C}\setminus\{p\}$, of special type of $V_{5}$ passing
through $q'$ is then $T$-stable with the point $q=\nu(r_{T})=\nu(r_{T}')\in\ell_{p}$
as its other $T$-stable point. 

$\quad$ - If $\Gamma\neq\Gamma_{p}$ then, by Proposition \ref{prop:T-stable-quintics},
$\Gamma$ intersects $\Gamma_{p}$ with multiplicity $4$ at $q'$.
Proposition \ref{prop:Hilbert-scheme-correspondence} 4) implies that
the proper transform $Z'$ of $\ell_{p'}$ in $X$ by $\psi_{Z}:X\dasharrow V_{5}$
is a $T$-stable line of special type of $X$ other than $Z$ and
that $Z$ and $Z'$ are the unique lines of special type on $X$. 

$\quad$ - If $\Gamma=\Gamma_{p}$ then, by Proposition \ref{prop:Hilbert-scheme-correspondence}
3), the lines in $X$ other than $Z$ are the proper transforms by
$\psi_{Z}:X\dasharrow V_{5}$ of the lines in $V_{5}$ not contained
in $F_{p}$ and meeting $\Gamma_{p}$. Since $\Gamma_{p}\subset\mathcal{S}$,
there are exactly two lines passing through a point of $\Gamma_{p}$
other than $p$: one line of general type contained in $F_{p}$ and
the other one of special type contained in $\mathcal{S}$. Using the
notation of subsection \ref{subsec:Lines-V5}, these lines of special
type are the images of the fibers of $\pi:\mathfrak{U}|_{\mathfrak{C}\setminus[\ell_{p}]}\to\mathfrak{C}\setminus[\ell_{p}]$
by the evaluation morphism $\psi:\mathfrak{U}\to V_{5}$. Proposition
\ref{prop:Hilbert-scheme-correspondence} 4) then implies that every
line $Z'$ in $X$ other than $Z$ is of special type, equal to the
image of a fiber of $\pi:\mathfrak{U}|_{\mathfrak{C}\setminus[\ell_{p}]}\to\mathfrak{C}\setminus[\ell_{p}]$
by the composition of $\psi:\mathfrak{U}\to V_{5}$ with $\psi_{Z}^{-1}$.
Similarly, given any such line $Z'$, the lines in $X$ other than
$Z'$ are the images of the fibers of $\pi:\mathfrak{U}|_{\mathfrak{C}\setminus[\ell_{p'}]}\to\mathfrak{C}\setminus[\ell_{p'}]$
by the composition $\psi_{Z'}^{-1}\circ\psi$. It follows that the
Hilbert scheme $\mathcal{H}_{X}$ of lines in $X$ is irreducible
and that $(\mathcal{H}_{X})_{\mathrm{red}}$ is covered by two affine
open subsets isomorphic to $\mathfrak{C}\setminus[\ell_{p}]$ and
$\mathfrak{C}\setminus[\ell_{p'}]$ intersecting along an open subset
isomorphic to $\mathfrak{C}\setminus\{[\ell_{p}],[\ell_{p'}]\}$.
Thus, $(\mathcal{H}_{X})_{\mathrm{red}}$ is isomorphic to $\mathbb{P}^{1}$,
and Proposition \ref{prop:Hilbert-scheme-correspondence} 4) then
says that all lines in $X$ are of special type. \\

We now show that for any $\Gamma=\Gamma^{T}(v)$, $v\in\mathbb{P}^{1}\setminus\{0,1,\infty\}$,
as above, there exists an involution $\theta$ of $X$ semi-commuting
with the action of $T$ -i.e. contained in the normalizer of $T$
in $\mathrm{Aut}(X)$ but not in its centralizer- and mapping $Z$
onto $Z'$. We fix an identification of $T$ with the neutral component
of the subgroup $\mathrm{Aut}(\mathcal{C},\{p,p'\})$ of $\mathrm{Aut}(\mathcal{C})$,
we let $F_{p'}$ be the unique hyperplane section singular along the
line $\ell_{p'}$. Since $\ell_{p}$ and $\ell_{p'}$ are lines of
special type, there exists a unique line $\ell_{q,q'}$ in $V_{5}$
passing through the $T$-stable points $q\in\ell_{p}\setminus\{p\}$
and $q'\in\ell_{p'}\setminus\{p'\}$. The line $\ell_{q,q'}$ is $T$-stable
and since $F_{p}$ and $F_{p'}$ are swept out by lines meeting $\ell_{p}$
and $\ell_{p'}$ respectively, $\ell_{q,q'}$ is contained in $F_{p}\cap F_{p'}$.
Since $F_{p}\cap F_{p'}$ is a $T$-stable hyperplane section of $F_{p}$,
it follows, with the notation of of Case (1) in subsection \ref{subsec:Classification},
that $F_{p}\cap F_{p'}$ is the union of $\ell_{q,q'}=\nu(f_{T})$
and of the $T$-stable curve $C_{q,q'}:=\nu(s_{T})$ passing through
$q$ and $q'$. 

Let $S_{Z}$ and $S_{Z'}$ be the unique hyperplane sections of $X$
of multiplicity $3$ along $Z$ and $Z'$ respectively, let $E_{Z}$
and $E_{Z'}$ be the exceptional divisors of the blow-up $\alpha_{Z,Z'}:X_{0}\to X$
of $X$ along $Z$ and $Z'$ respectively and let $R_{\{Z,Z'\}}=(\psi_{Z}^{-1})_{*}F_{\ell_{p'}}$
be the hyperplane section of $X$ having multiplicity $2$ along $Z$
and $Z'$. We view all these as $T$-stable divisors over $X$. By
definition of $Z'$ and $\psi_{Z}$, $(\psi_{Z}\circ\alpha_{Z,Z'})_{*}E_{Z}=F_{p}$,
$(\psi_{Z}\circ\alpha_{Z,Z'})_{*}E_{Z'}=\ell_{p'}$, $(\psi_{Z})_{*}S_{Z}=\Gamma$
whereas $(\psi_{Z})_{*}S_{Z'}$ is the cubic section of $X$ having
multiplicity $3$ along $\ell_{p'}$ and multiplicity $2$ along $\Gamma$. 

Let $\pi_{\ell_{p'}}:V_{5}\dashrightarrow Q$, where $Q\subset\mathbb{P}(H^{0}(V_{5},\mathcal{O}_{\mathbb{P}(W)}(1)|_{V_{5}}\otimes\mathcal{I}_{\ell_{p'}}))=\mathbb{P}^{4}$
is a smooth quadric threefold, be the birational map defined by the
linear system of hyperplane sections of $V_{5}$ containing $\ell_{p'}$,
see \cite{Fuj81}. Since $\ell_{p'}$ is $T$-stable, $\pi_{\ell_{p'}}$
is $T$-equivariant for the faithful $T$-action on $Q$ induced by
the $T$-module structure on $H^{0}(V_{5},\mathcal{I}_{\ell_{p'}}\otimes\mathcal{O}_{\mathbb{P}(W)}(1)|_{V_{5}})$,
and the composition $\xi:=\pi_{\ell_{p'}}\circ\psi_{Z}:X\dashrightarrow Q$
is then $T$-equivariant as well. The line $\ell_{q,q'}$ is contracted
by $\pi_{\ell_{p'}}$ onto a $T$-fixed point $r'\in Q$. This point
is contained in the image $\Gamma^{Q}$ of $\Gamma$, which is a $T$-stable
rational normal quartic curve in $Q$ having $r:=\pi_{\ell_{p'}}(p)$
as its second $T$-fixed point. The image $\ell_{p}^{Q}$ of $\ell_{p}$
is the tangent line to $\Gamma^{Q}$ at $r$. The proper transforms
$E_{Z}^{Q}$ and $E_{Z'}^{Q}$ of $E_{Z}$ and $E_{Z'}$ on $Q$ are
respectively equal to a $T$-stable quadric section of $Q$ containing
$\Gamma^{Q}$ and singular along $\ell_{p}^{Q}$ and the tangent hyperplane
section to $Q$ at the point $r'$. Moreover, the image $C_{q,q'}^{Q}$
of $C_{q,q'}$ is a $T$-stable rational normal cubic curve, equal
to the image of $F_{\ell_{p'}}$ by $\pi_{\ell_{p'}}$, contained
in $E_{Z'}^{Q}\cap E_{Z}^{Q}$ and intersecting $\Gamma^{Q}$ with
multiplicity $3$ at $r'$. On the other hand, the proper transform
$S_{Z'}^{Q}$ of $S_{Z'}$ is the $T$-stable cubic section of $Q$
singular along $\Gamma^{Q}$. By Lemma \ref{prop:The-crucial-involution},
there exists a birational involution $\iota_{Q}:Q\dashrightarrow Q$
semi-commuting with the action of $T$ on $Q$, contracting $S_{Z'}^{Q}$
onto $\Gamma^{Q}$ and mapping $C_{q,q'}^{Q}$ isomorphically onto
itself and such that $(\iota_{Q}^{-1})_{*}E_{Z'}^{Q}=E_{Z}^{Q}$.
The conjugate $\theta:=\xi^{-1}\circ\iota_{Q}\circ\xi:X\dashrightarrow X$
is a birational involution of $X$ mapping $Z$ onto $Z'$, $S_{Z}$
onto $S_{Z'}$ and whose unique possibly contracted divisor is the
hyperplane section $R_{\{Z,Z'\}}=(\psi_{Z}^{-1})_{*}F_{\ell_{p'}}$.
But since $\iota_{Q}$ maps $C_{q,q'}^{Q}$ isomorphically onto itself
by Lemma \ref{prop:The-crucial-involution} 2) and 
\[
\xi_{*}R_{\{Z,Z'\}}=(\pi_{\ell_{p'}})_{*}F_{\ell_{p'}}=C_{q,q'}^{Q}\subset E_{Z'}^{Q}\cap E_{Z}^{Q}=E_{Z'}^{Q}\cap(\iota_{Q}^{-1})_{*}E_{Z'}^{Q},
\]
it follows that $\theta$ does not contract any divisor, hence is
a biregular involution of $X$.\\

We can now conclude the description of the isomorphism types and automorphism
groups of the threefolds $X$. 

$\quad$ - If $\Gamma=\Gamma^{T}(v)\neq\Gamma_{p},$ then $\mathrm{Aut}^{0}(X)=\mathrm{Aut}(X,Z)=\mathrm{Aut}(X,Z')=T$
is a subgroup of index $2$ of $\mathrm{Aut}(X)=\mathrm{Aut}(X,Z\cup Z')$
whose quotient is generated by the image of the involution $\theta$
exchanging $Z$ and $Z'$. Since this involution semi-commutes with
the action of $\theta$, we have $\mathrm{Aut}(X)\cong T\rtimes\langle\theta\rangle\cong\mathbb{G}_{m}\rtimes\mu_{2}$.
The existence of $\theta$ implies that the pairs $(X,Z)$ and $(X,Z')$
are isomorphic, whence by Proposition \ref{prop:T-stable-quintics}
and Corollary \ref{cor:T-stable-quintics}, that the threefolds $X_{22}^{m}(v)$,
$v\in\mathbb{P}^{1}\setminus\{0,1,\infty,-4\}$ are pairwise non isomorphic. 

$\quad$ - If $\Gamma=\Gamma_{p}$ then $\mathrm{Aut}(X,Z)=G_{p}$
and $\mathrm{Aut}(X,Z')\subset\mathrm{Aut}(X)$ is the conjugate of
$\mathrm{Aut}(X,Z)$ by the involution $\theta$. Since the latter
normalizes $T$, $\mathrm{Aut}(X,Z)$ and $\theta$ generate a subgroup
of $\mathrm{Aut}(X)$ isomorphic to $\mathrm{PGL}_{2}$, from which
it follows, by Lemma \ref{prop:line}, that $\mathrm{Aut}^{0}(X)\cong\mathrm{PGL}_{2}$.
Finally, since an element of $\mathrm{Aut}(X)$ which stabilizes every
line in $X$ is contained in the intersection of all Borel subgroups
of $\mathrm{Aut}^{0}(X)$, it must be trivial, which implies that
the canonically induced action of $\mathrm{Aut}(X)$ on $(\mathcal{H}_{X})_{\mathrm{red}}\cong\mathbb{P}^{1}$
is faithful. Thus, $\mathrm{Aut}(X)$ identifies with a subgroup of
the automorphism group of $(\mathcal{H}_{X})_{\mathrm{red}}\cong\mathbb{P}^{1}$,
hence is connected, equal to $\mathrm{PGL}_{2}$. 
\end{proof}
In the proof of Theorem \ref{prop:Second-Isomorphism-Classification}
above, we have used the following auxiliary result on the birational
geometry of smooth quadric threefolds with an effective action of
a $1$-dimensional torus. 
\begin{lem}
\label{prop:The-crucial-involution}Let $\Gamma_{4}\subset\mathbb{P}^{4}$
be a rational normal quartic curve and let $T\subset\mathrm{Aut}(\mathbb{P}^{4},\Gamma_{4})\cong\mathrm{Aut}(\Gamma_{4})\cong\mathrm{PGL}_{2}$
be a maximal torus. Then the following hold:

1) For every smooth $T$-stable quadric $Q$ containing $\Gamma_{4}$,
the linear system of quadric sections of $Q$ containing $\Gamma_{4}$
determines a birational involution $\iota_{Q}:Q\dashrightarrow Q$
contracting the unique cubic section $S_{\Gamma_{4}}$ of $Q$ singular
along $\Gamma_{4}$ onto $\Gamma_{4}$ and restricting to an automorphism
of $Q\setminus S_{\Gamma_{4}}$ which semi-commutes with the induced
$T$-action on it.

2) Let $r,r'\in\Gamma_{4}$ be the two $T$-fixed points. Then the
proper transform $(\iota_{Q}^{-1})_{*}H_{r'}$ of the tangent hyperplane
section $H_{r'}$ to $Q$ at $r'$ is the unique integral quadric
section of $Q$ containing $\Gamma_{4}$ and singular along the tangent
line $L_{r}$ to $\Gamma_{4}$ at $r$. The scheme $H_{r'}\cap(\iota_{Q}^{-1})_{*}H_{r'}$
is the union of the line $L_{r',s}$ passing through $r'$ and the
unique $T$-fixed point $s\in L_{r}\setminus\{r\}$ and of a rational
normal cubic curve $C_{2r',s}$ with tangent line $L_{r'}$ at $r'$
and passing through $s$. Moreover, $\iota_{Q}$ maps $C_{2r',s}$
isomorphically onto itself.
\end{lem}

\begin{proof}
For any smooth quadric $Q$ containing $\Gamma_{4}$, the linear system
of quadric sections of $Q$ containing $\Gamma_{4}$ defines a birational
map $Q\dashrightarrow Q'$ onto a smooth quadric threefold $Q'\subset\mathbb{P}^{4}$
which contracts the unique cubic section $S_{\Gamma_{4}}$ of $Q$
singular along $\Gamma_{4}$ onto a rational normal quartic curve
$\Gamma_{4}'\subset Q'$ and restricts to an isomorphism $Q\setminus S_{\Gamma_{4}}\to Q'\setminus S_{\Gamma_{4}'}$,
see e.g \cite[Lemma 2.2]{KP22}. The fact that for every smooth $T$-stable
quadric $Q$ containing $\Gamma_{4}$, the pairs $(Q,\Gamma_{4})$
and $(Q',\Gamma_{4}')$ are isomorphic follows from \cite[$\S$ 5.9]{Calabi23},
\cite[$\S$ 9]{CPS19}. We now argue that we can choose explicit coordinates
on $\mathbb{P}^{4}$ so to obtain from these existing description
a birational involution $\iota_{Q}:Q\dashrightarrow Q$ with the desired
properties. 

We can assume without loss of generality that $\Gamma_{4}$ is the
image of the closed immersion $\mathbb{P}_{[t_{0}:t_{1}]}^{1}\to\mathbb{P}^{4}$,
$[t_{0}:t_{1}]\mapsto[t_{1}^{4}:t_{1}^{3}t_{0}:t_{1}^{2}t_{0}^{2}:t_{1}t_{0}^{3}:t_{0}^{4}]$
and that $T=\mathrm{Spec}(k[\lambda^{\pm1}])$ acts on $\mathbb{P}_{[t_{0}:t_{1}]}^{1}$
by $\lambda\cdot[t_{0}:t_{1}]=[\lambda t_{0}:t_{1}]$. Then $H^{0}(\mathbb{P}^{4},\mathcal{I}_{\Gamma_{4}}\otimes\mathcal{O}_{\mathbb{P}^{4}}(2))$,
where $\mathcal{I}_{\Gamma_{4}}\subset\mathcal{O}_{\mathbb{P}^{4}}$
is the ideal sheaf of $\Gamma_{4}$, is generated by the six polynomials
\[
f_{6}=w_{2}w_{4}-w_{3}^{2},f_{5}=w_{1}w_{4}-w_{2}w_{3},f_{4,0}=w_{0}w_{4}-w_{2}^{2},f_{4,1}=w_{1}w_{3}-w_{2}^{2},f_{3}=w_{0}w_{3}-w_{1}w_{2},f_{2}=w_{0}w_{2}-w_{1}^{2}
\]
and the $T$-stable smooth quadrics containing $\Gamma_{4}$ are those
$Q_{[a:b]}$ defined by the polynomials 
\[
f_{[a:b]}=af_{4,0}-bf_{4,1}=a(w_{0}w_{4}-w_{2}^{2})-b(w_{1}w_{3}-w_{2}^{2}),\quad[a:b]\in\mathbb{P}^{1}\setminus\{0,1,\infty\}.
\]
Given such a quadric $Q=Q_{[a:b]}$, choosing $c$ so that $c^{2}=\tfrac{a}{b}$,
it is straightforward to verify that the map 
\begin{equation}
\psi:Q\dashrightarrow\mathbb{P}^{4},\,[w_{0}:w_{1}:w_{2}:w_{3}:w_{4}]\mapsto[f_{2}:cf_{3}:c^{2}f_{4,0}:cf_{5}:f_{6}]\label{eq:The-involution}
\end{equation}
 is $T$-equivariant, has image contained in $Q$ and that the induced
birational map $j_{Q}:Q\dashrightarrow Q$ is an involution. On the
other hand, the pair $(Q,\Gamma_{4})$ is stable under the involution
$i$ of $\mathbb{P}^{4}$ defined by $[w_{0}:w_{1}:w_{2}:w_{3}:w_{4}]\mapsto[w_{4}:w_{3}:w_{2}:w_{1}:w_{0}]$
which semi-commutes with the action of $T$. The induced involution
$i_{Q}=i|_{Q}$ commutes with $j_{Q}$, and $\iota_{Q}:=i_{Q}\circ j_{Q}=j_{Q}\circ i_{Q}$
is the desired involution. 

To prove assertion 2), we can assume that $r=[0:0:0:0:1]$ and $r'=[1:0:0:0:0]$
so that $H_{r'}=\{w_{4}=0\}\cap Q$, $L_{r'}=\{w_{2}=w_{3}=w_{4}=0\}$
and $L_{r}=\{w_{0}=w_{1}=w_{2}=0\}$. Then $(\iota_{Q}^{-1})_{*}H_{r'}=\{w_{0}w_{2}-w_{1}^{2}=0\}\cap Q$
is an anti-canonically embedded del Pezzo surface of degree $4$ singular
along $L_{r}$ and $H_{r'}\cap(\iota_{Q}^{-1})_{*}H_{r'}=\{w_{4}=w_{0}w_{2}-w_{1}^{2}=0\}\cap Q$
is the union of the line $L_{r',s}=\{w_{1}=w_{2}=w_{4}=0\}$ where
$s=[0:0:0:1:0]\in L_{r}\setminus\{r\}$ and of the rational normal
cubic curve $C_{2r',s}$ with tangent line $L_{r'}$ at $r'$ passing
through $s$ given by the image of the morphism 
\[
\alpha:\mathbb{P}_{[u_{0}:u_{1}]}^{1}\to\mathbb{P}^{4},\,[u_{0}:u_{1}]\mapsto[u_{0}^{3}:u_{0}^{2}u_{1}:u_{0}u_{1}^{2}:(1-c^{2})u_{1}^{3}:0].
\]
Now it is routine to verify from (\ref{eq:The-involution}) that $\iota_{Q}\circ\alpha=\alpha\circ\iota_{c}$,
where $\iota_{c}$ is the involution $[u_{0}:u_{1}]\mapsto[u_{1}:\frac{c}{1-c^{2}}u_{0}]$
of $\mathbb{P}^{1}$. 
\end{proof}
\begin{rem}
\label{rem:Link-remark}In the proof of Theorem \ref{prop:Second-Isomorphism-Classification},
we constructed for a threefold $X=X_{22}^{m}(v)$, $v\in\mathbb{P}^{1}\setminus\{0,1,\infty\}$,
where $X_{22}^{m}(-4)=X_{22}^{MU}$, and a pair $(Z,Z')$ of $\mathbb{G}_{m}$-stable
lines of special type in $X$, a biregular involution $\theta$ of
$X$ semi-commuting with the $\mathbb{G}_{m}$-action and exchanging
the lines $Z$ and $Z'$. Let us briefly indicate a complementary
interpretation of the construction of this involution $\theta$. With
the notation of the proof of Theorem \ref{prop:Second-Isomorphism-Classification},
the composition of the double projection $\psi_{Z}:X\dashrightarrow V_{5}$
from $Z$ with the projection $\pi_{(\psi_{Z})_{*}Z'}:V_{5}\dashrightarrow Q$
from the line $(\psi_{Z})_{*}Z'$ is a $\mathbb{G}_{m}$-equivariant
birational map which contracts the unique hyperplane section $S_{Z}$
of $X$ of multiplicity $3$ along $Z$ onto a rational normal quartic
curve $\Gamma^{Q}$, contracts the unique hyperplane section $R_{\{Z,Z'\}}$
of $X$ singular along $Z\cup Z'$ onto a rational normal cubic curve
$C$ and maps the the unique hyperplane section $S_{Z'}$ of $X$
singular along $Z'$ onto a cubic section $S_{Z'}^{Q}$ of $Q$ singular
along $\Gamma^{Q}$. Letting $\alpha_{\Gamma^{Q}}:\hat{Q}\to Q$ be
the blow-up of $\Gamma^{Q}$, where $\hat{Q}\subset\mathbb{P}^{16}$
is a Fano threefold of Picard rank $2$ and degree $28$ of type 2.21
in \cite{MoMu81}, we obtain an induced $\mathbb{G}_{m}$-equivariant
birational map 
\[
\hat{\gamma}=\alpha_{(\pi_{(\psi_{Z})_{*}Z'}\circ\psi_{Z})_{*}S_{Z}}^{-1}\circ\pi_{(\psi_{Z})_{*}Z'}\circ\psi_{Z}:X\dashrightarrow\hat{Q}
\]
which fits into the following commutative diagram of $\mathbb{G}_{m}$-equivariant
birational maps
\[
\xymatrix{ &  & R'_{\{Z,Z'\}}\subseteq Y\ar@{.>}[r]^{\hat{\chi}}\ar@{->}[lld]_{\hat{\sigma}} & Y^{+}\supseteq R^{+}{}_{\{Z,Z'\}}\ar@{->}[rrd]^{\hat{\tau}}\\
Z\cup Z'\subset X\ar@{-->}[rrrrr]^{\hat{\gamma}} &  &  &  &  & \hat{Q}\supset\hat{C}%
\\
}
\]
where $\hat{\sigma}:Y\to X$ is the blow-up of $Z\cup Z'$, $\hat{\chi}:Y\dashrightarrow Y^{+}$
is a small $\mathbb{Q}$-factorial modification and $\hat{\tau}:Y^{+}\to\hat{Q}$
is the contraction of the proper transform $R^{+}{}_{\{Z,Z'\}}$ of
$R_{\{Z,Z'\}}$ onto a smooth curve $\hat{C}\subset\hat{S}$ of bi-degree
$(3,3)$. One can further check from the construction that $\hat{\gamma}$
is given by the linear system of quadric sections of $X$ having multiplicity
$3$ along $Z$ and $Z'$. The birational involution $\iota_{Q}:Q\dashrightarrow Q$
in the proof of Theorem \ref{prop:Second-Isomorphism-Classification}
lifts to a biregular involution $\iota_{\hat{Q}}$ of $\hat{Q}$ which
normalizes the lifted $\mathbb{G}_{m}$-action and makes $\hat{Q}$
into a Fano threefold of $(\mathbb{G}_{m}\rtimes\mu_{2})$-Picard
rank $1$, and by the construction of the associated biregular involution
$\theta$ of $X$, the above diagram is then a $(\mathbb{G}_{m}\rtimes\mu_{2})$-equivariant
Sarkisov link. We refer the reader to the forthcoming article \cite{DFK-next}
for a for a more in-depth and detailed study of these birational links
between smooth prime Fano threefolds of degree 22 and blow-ups of
smooth quadric threefolds along rational normal quartic curves. 
\end{rem}

\bibliographystyle{amsplain}

\end{document}